\documentclass{article}

\usepackage{lipsum}
\usepackage{amsfonts}%
\usepackage{graphicx}
\usepackage{algorithm}
\usepackage{algpseudocode}  
\usepackage{amsmath}%
\usepackage{amssymb}%
\usepackage{bbm}
\usepackage{color}
\usepackage{esint}
\usepackage{booktabs}
\usepackage{stmaryrd}
\usepackage{amsthm}
\usepackage{subcaption}

\usepackage{authblk}

\usepackage{enumitem}
\setlist[enumerate]{leftmargin=.5in}
\setlist[itemize]{leftmargin=.5in}

\newtheorem{remark}{Remark}

\newtheorem{assumption}{Assumption}
\newtheorem{theorem}{Theorem}
\newtheorem{definition}{Definition}
\newtheorem{lemma}{Lemma}
\newtheorem{corollary}{Corollary}
\newtheorem{proposition}{Proposition}

\usepackage{amsopn}

\renewcommand{\(}{\left(}
\renewcommand{\)}{\right)}

\author{Chengyang Liu\thanks{Department of Mathematics, The University of Hong Kong ({cyliu00@connect.hku.hk}). C. Liu’s research was supported by a Postgraduate Scholarship (PGS).}, Michael K. Ng\thanks{Department of Mathematics, Hong Kong Baptist University ({michael-ng@hkbu.edu.hk}). M. Ng’s research was supported by the GDSTC: Guangdong and Hong Kong Universities “1+1+1” Joint Research Collaboration Scheme UICR0800008-24, National Key Research and Development Program of China under Grant 2024YFE0202900, RGC GRF 12300125 and Joint NSFC and RGC N-HKU769/21.}}

\title{A Bijective Image Retargeting Algorithm Based on Conformal Energy\thanks{Submitted to the editors DATE.
}}

\providecommand{\keywords}[1]
{
  \small	
  \textbf{\text{Keywords: }} #1
}

\begin{document}

\captionsetup[figure]{labelfont={bf},name={Fig.},labelsep=period}

\maketitle

\begin{abstract}
 Image retargeting, which resizes images to one with a prescribed aspect ratio by determining an optimal warping map, has gained substantial interest in imaging science. Despite significant advances, existing methods often fail to ensure bijective warping maps essential for preserving visual information. This paper introduces a novel bijective image retargeting model through conformal energy minimization of the deformation field. The proposed model establishes mathematical rigor by proving well-posedness for the optimal warping map in both continuous and discrete settings and showing that the discrete solutions converge to their continuous counterpart under mesh refinement. Numerical experiments corroborate the model's efficacy and the convergence of discrete solutions during progressive mesh subdivision processes, validating both theoretical guarantees and practical performance. 
\end{abstract}


\section{Introduction}
Nowadays, an increasing variety of electronic devices feature diverse screen aspect ratios. The problem of adapting images to various aspect ratios is called image retargeting. It may cause visual unnaturality and geometric information loss of the resulting images, by displaying one picture on different devices through simple stretching to the desired ratio. This situation gives rise to the need for more new image retargeting algorithms, which can resize images into different sizes while preserving some properties of the original images.

\subsection{Review and contributions}
\label{sec:review}
{
Image retargeting aims to adapt an image to a different aspect ratio while preserving its important content, limiting visual artifacts in the resulting media, and preserving internal structures of the original media \cite{rubinstein2010comparative}. Over the years, a multitude of algorithms have been proposed, which can be broadly classified into three primary paradigms: cropping-based, warping-based, and more recently, learning-based methods.}

{Cropping-based methods represent the most direct approach, achieving the target aspect ratio by removing pixels from the image and resetting the other pixels, e.g., Seam Carving algorithm \cite{avidan2007seam}, image retargeting algorithm using nonparametric semantic
segmentation \cite{razzaghi2015image}, shift-map image algorithm \cite{chen2016automatic}, and automatic rectangle cropping algorithm \cite{garg2022improved}.
 Taking Seam Carving as an example, it iteratively removes the least important seams. However, this process may destroy critical objects if seams inadvertently pass through essential areas. Fundamentally, the mapping in such methods is non-bijective, leading to irreversible information loss and the risk of removing important contextual elements.}

{Warping-based retargeting methods seek a continuous and smooth mapping, or a deformation field, to transform a source domain into a target one. 
A common strategy is to discretize the image domain with a quadrilateral \cite{wolf2007nonhomogeneous, niu2012image, kim2018quad} or triangular mesh \cite{lau2018image, xu2018content} and then compute the new vertex positions to approximate the desired warping map. 
Among these methods, conformal and quasi-conformal mappings are particularly effective for retargeting, as they provide excellent control over local geometric distortion. Recent quasi-conformal approaches \cite{lau2018image, xu2018content} rely on the linear Beltrami solver to efficiently compute the warping map from its associated Beltrami coefficient by solving a linear system. Despite its computational efficiency, this solver does not theoretically guarantee the bijectivity of the warping map, potentially leading to severe artifacts such as self-intersections. Therefore, a central challenge in warping-based retargeting remains the formulation of a model that is both computationally tractable and theoretically guarantees bijectivity.}

{
Learning-based methods mark a significant evolution in image retargeting. Their primary advantage lies in the ability to generate novel, contextually coherent pixels.  DeepIR \cite{lin2019deepir} performs pixel generating within a deep feature space. InGAN \cite{shocher2019ingan} demonstrates this by learning an image's internal patch distribution to synthesize new content. Vision Transformers like NaViT \cite{dehghani2023navit} introduce a "Patch n' Pack" mechanism that enables the network to process batches of images with arbitrary native resolutions and aspect ratios. However, this generative power comes with the significant drawback of potential "artificial hallucinations".
}

{
To address the central challenge of guaranteeing bijectivity in warping-based methods, this paper proposes a bijective image retargeting algorithm. 
In this paper, we propose a bijective image retargeting algorithm that can minimize the conformal energy of warping maps in the retargeting process while preserving the shapes of regions of interest and line structures in images. 
Then, the output of the energy-minimizing algorithm will be smoothed by the bijection correction algorithm, so the final warping map is bound to be bijective. Our contributions include a detailed theoretical analysis of the proposed model, establishing its mathematical soundness.
}

\subsection{Notation and outline}
The following notation is used in the paper. Other notation will be clearly defined whenever they appear.
\begin{itemize}
    \item Mathcal font letters, e.g., $\mathcal{D}_1$, $\mathcal{V}$, $\mathcal{O}$, denote various sets;
    \item Bold letters, e.g., $\mathbf{x}$, $\mathbf{v}$, $\mathbf{t}_{\mathcal{O}_i}$, denote vectors.
    \item $\tau = \left[ \mathbf{v}_0, ..., \mathbf{v}_l\right]$ denotes the $l$-simplex.
    \item $\left( L \right)_{i,j}$ denotes the $\left(i, j\right)$th entry of the matrix $L$.
\end{itemize}

This paper is organized as follows: In Section \ref{sec:1}, we build one mathematical model to illustrate the retargeting problem and introduce the conformal energy as the measure of local geometric distortions. In Section \ref{sec:2}, we discretize the retargeting problem into one optimization problem whose feasible region is in the simplicial mapping domain, and we also propose the bijection correction algorithm and an alternative initialization algorithm for the parameters and the boundary condition in the retargeting model. In Section \ref{sec:3}, we show the existence and uniqueness of the optimization problem in the continuous and discrete setting, and the optimal simplicial mapping in the discrete setting can be a promising approximation for the optimal mapping in the continuous setting if the mesh is dense enough; we also show that our main algorithm can ensure the final warping map is bijective. In Section \ref{sec:4}, we give some experimental results of our main algorithm and do some comparisons with other retargeting methods. All the proofs are put in appendices.

\section{Background}
\label{sec:1}
\subsection{Mathematical modeling}
We build up one mathematical model to illustrate the image retargeting problems.

Regions of interest (ROIs), usually the main objects in one image, are the critical parts that attract most people's attention at first sight.
Therefore, preserving the shapes of those objects entirely during the retargeting process is reasonable. 
Line structures are another important factor in images. The human visual system is sensitive to distortions except stretching, so we set the warping map on line structures as a simple stretching mapping with translation.

In this work, we focus on adjusting the image width to achieve the target aspect ratio. Denote that the region of original image is $\mathcal{D}_1 = [0,a] \times [0,b] \subset \mathbb{R}^2$,
and the region of output of retargeting algorithm is $\mathcal{D}_2 = [0,wa] \times [0,b]\subset \mathbb{R}^2$, where $w$ is the resizing ratio.
Suppose an image is defined as a mapping $I_1: \mathcal{D}_1 \rightarrow \mathbb{R}$. 
The target image is $I_2: \mathcal{D}_2 \rightarrow \mathbb{R}$.
The ROIs in $\mathcal{D}_1$ are $\{ \mathcal{O}_i \subseteq \mathcal{D}_1 \}^K_{i=1} $ and line structures are $\{ \mathcal{L}_j \subseteq \mathcal{D}_1\}^L_{j=1}$. 
Denote $\mathcal{O} := \bigcup_{i=1}^k \mathcal{O}_i$, $\mathcal{L} := \bigcup_{j=1}^l \mathcal{L}_j$, 
and $\Omega := \operatorname{int}\(\mathcal{D}_1 \backslash \(\mathcal{O} \cup \mathcal{L}\)\)$. 
Assume that $\mathcal{O}_1, \mathcal{O}_2, \dots, \mathcal{O}_K, \mathcal{L}_1, \dots, \mathcal{L}_L, \partial \mathcal{D}_1$ are disjoint sets.

The retargeting map is a bijective warping map $f: \mathcal{D}_1 \rightarrow \mathcal{D}_2$. The ROIs and line constraints can be illustrated in two formulas below,
\begin{equation}
  \begin{aligned}
    f \mid_{\mathcal{O}_i}\(\mathbf{x}\) &= r_{\mathcal{O}} \mathbf{x} + \mathbf{t}_{\mathcal{O}_i} \in \mathcal{D}_2,\\
    f \mid_{\mathcal{L}_j}\(\mathbf{x}\) &= R_{\mathcal{L}_j} \mathbf{x} + \mathbf{t}_{\mathcal{L}_j} \in \mathcal{D}_2,
  \end{aligned}
\end{equation}
in which $r_{\mathcal{O}}\in \mathbb{R}^+$, 
$R_{\mathcal{L}_j} = \(\begin{matrix}
  r^x_{\mathcal{L}_j} & \\
     &  r^y_{\mathcal{L}_j}
\end{matrix} \)$ is a 2$\times$2 diagonal matrix, whose diagonal elements are all positive, and $\mathbf{t}_{\mathcal{O}_i} = \(\begin{matrix}{t}^x_{\mathcal{O}_i} \\ {t}^y_{\mathcal{O}_i}\end{matrix}\), \mathbf{t}_{\mathcal{L}_j} = \(\begin{matrix}{t}^x_{\mathcal{L}_j}\\ {t}^y_{\mathcal{L}_j}\end{matrix}\) \in \mathbb{R}^2$, $i=1, 2, \dots, k$, $j=1, 2, \dots, l$.
These parameters are given in advance. For the boundary constraint, the boundary condition $g$ should be a bijective continuous map, which satisfies the boundary conditions $g\(0, 0\) = \(0, 0\), g\(a, 0\) = \(wa, 0\), g\(0, b\) = \(0, b\)$, and $g\(a, b\) = \(wa, b\)$.

Since the geometric distortion in $\Omega$ is invertible, we want to make the distortion distribution over $\Omega$ as even as possible.
Therefore, we introduce the conformal energy to our model to achieve this goal.

\subsection{Conformal Energy}
Given a bijective map $f: \mathcal{D}_1 \rightarrow \mathcal{D}_2,\, \left(x,y\right)\mapsto \left(u,v\right)$ such that $f\left(\mathcal{D}_1\right)=\mathcal{D}_2$ and $f\left(\partial \mathcal{D}_1\right) = \partial \mathcal{D}_2$, conformal energy functional $E^C\left( f\right)$ is the measure of local geometric distortion of function $f$. 
Conformal energy is defined by
\begin{equation}
  \begin{aligned}
  E^C\(f\) & =\frac{1}{2} \int_{\mathcal{D}_1}\left(\frac{\partial u}{\partial x}-\frac{\partial v}{\partial y}\right)^2+\left(\frac{\partial u}{\partial y}+\frac{\partial v}{\partial x}\right)^2 \\
  =&\frac{1}{2} \int_{\mathcal{D}_1}{\|\nabla f\|^2}-\operatorname{Area}\(f\(\mathcal{D}_1\)\)\\
  =& E^D\(f\) - \operatorname{Area}\(\mathcal{D}_2 \)\ge 0, 
  \end{aligned}
\end{equation}
where {$\|\nabla f\|$ represents the Frobenius norm of $\nabla f$,} $E^D \left(f\right) $ is the Dirichlet energy of $f$, and $\operatorname{Area}:  \omega \mapsto \int_{\omega}{dxdy}$ is the operator to map the set $\omega$ to its area. 
Especially, $E^C(f) = 0$ if and only if $f$ is a conformal mapping.

\section{The Bijective Retargeting Algorithm}
\label{sec:2}
\subsection{Minimizer in the continuous setting}
In the retargeting problem, the feasible region $\mathcal{R}$ of $E^C$ is set to be $\mathcal{R} =\{f \in H^1\(\mathcal{D}_1\) \mid f|_{\partial \mathcal{D}_1}(\mathbf{x}) = g(\mathbf{x});\, f \mid_{\mathcal{O}_i}\(\mathbf{x}\) = r_{\mathcal{O}} \mathbf{x} + \mathbf{t}_{\mathcal{O}_i}, i=1,\dots, k;\,f \mid_{\mathcal{L}_j}\(\mathbf{x}\) = R_{\mathcal{L}_j} \mathbf{x} + \mathbf{t}_{\mathcal{L}_j}, j=1,\dots,l\}$.
The parameters $r_{\mathcal{O}}, \mathbf{t}_{\mathcal{O}_i}, R_{\mathcal{L}_j}, \mathbf{t}_{\mathcal{L}_j}, g$ are all determined in advance. 
The optimization problem can be illustrated by the following formulas:
\begin{equation}
  \label{op:continuous}
  \begin{aligned}
    \min _{f\in H^1\left( \mathcal{D}_1 \right)}& -\operatorname{Area} \left( \mathcal{D}_2 \right)+\frac{1}{2}\int_{\mathcal{D}_1}{{\|\nabla f\|}^2}\\
    \text{s.t. } &f\mid _{\partial \mathcal{D}_1}=g,\\
    &f\mid _{\mathcal{O}_i}\left( \mathbf{x} \right) =r_{\mathcal{O}}\mathbf{x}+\mathbf{t}_{\mathcal{O}_i},\, i = 1, 2, \dots, k,\\
    &f\mid _{\mathcal{L}_j}\left( \mathbf{x} \right) =R_{\mathcal{L}_j}\mathbf{x}+\mathbf{t}_{\mathcal{L}_j},\, j = 1, 2, \dots, l.\\
  \end{aligned}
\end{equation}
Denote the minimizer in the continuous setting as $f^* \in \mathcal{R}$.

\subsection{Discretization}
In the image retargeting problem, $\mathcal{D}_1$ and $\mathcal{D}_2$ are two rectangles with different aspect ratios.
We discretize $\mathcal{D}_1$ with simply connected Delaunay triangulation mesh $M = \left<\mathcal{V}, \mathcal{E}, \mathcal{F}\right>$ to make it computable \cite{gu2008computational}.
$$
\mathcal{V}(M)=\left\{\mathbf{v}_{\ell}=\left(x_{\ell}, y_{\ell}\right) \in \mathcal{D}_1,\, \ell = 1, 2, \dots, n_v \right\}
$$
is the set of vertices of the mesh, $\mathcal{E}(M)$ is defined to be the edges set, and
$$
\mathcal{F}(M)=\left\{ \tau_{\ell} = \left[\mathbf{v}_{i_{\ell}}, \mathbf{v}_{j_{\ell}}, \mathbf{v}_{k_{\ell}}\right] \subseteq \mathcal{D}_1,\, \ell = 1, 2, \dots, n_f \right\}
$$
denotes the set of the simplices in the mesh, where 
$$
\begin{aligned}
   \left[\mathbf{v}_{i_{\ell}}, \mathbf{v}_{j_{\ell}}, \mathbf{v}_{k_{\ell}}\right]:= 
  \left\{\mathbf{v}^{\mathbf{\alpha}}_{\ell} = \alpha_i \mathbf{v}_{i_{\ell}}+\alpha_j \mathbf{v}_{j_{\ell}}+\alpha_k \mathbf{v}_{k_{\ell}} \mid
  \right.\left.
   \alpha_i+\alpha_j+\alpha_k=1; \alpha_i, \alpha_j, \alpha_k \geq 0 \right\}.
\end{aligned}
$$
The triangular faces cover the domain of mapping $f$, which means $\mathcal{D}_1 = \left|\sum\limits_{\ell=1}^{n_f} \tau_{\ell}\right|$ and $\partial \mathcal{D}_1 = |\partial M|$, where $|C|:=\{x\mid x\in\tau, \,\tau \text{ is a simplex in complex } C\}$.

Next, we define the discretized version of $\mathcal{O}$, $\mathcal{L}$, $\partial \mathcal{D}_1$, and $\Omega$. For ROIs part, we denote $$\mathcal{O}_i\(M\) = \left\{\mathbf{v}_{\ell} \mid \mathbf{v}_{\ell}\in \tau, \tau \in \mathcal{F}\(M\), \tau \cap \mathcal{O}_i \neq \emptyset \right\},$$ and 
$$\mathcal{O}\(M\) = \bigcup_{i=1}^k \mathcal{O}_i\(M\).$$
As for line structures, 
$$
\begin{aligned}
  \mathcal{L}_j\(M\) = \{\mathbf{v}_{\ell} \mid \mathbf{v}_{\ell}\in \tau \in \mathcal{F}\(M\), \operatorname{int}\(\tau\) \cap \mathcal{L}_j \neq \emptyset\}\\
  \cup \{\mathbf{v}_{\ell} \mid \mathbf{v}_{\ell}\in \left[\mathbf{v}_{\ell}, \mathbf{v}_{k}\right]\in \mathcal{E}\(M\), \left[\mathbf{v}_{\ell}, \mathbf{v}_{k}\right] \cap \mathcal{L}_j \neq \emptyset\},
\end{aligned}
$$ 
and 
$$\mathcal{L}\(M\) = \bigcup_{j=1}^L \mathcal{L}_j\(M\).
$$
The set of boundary vertices is $$\mathcal{B}\(M\) = \{\mathbf{v}_{\ell}\mid \mathbf{v}_{\ell}\in \partial \mathcal{D}_1\}.
$$
The set of remaining vertices is
$$\mathcal{I}\(M\) = \mathcal{V}\(M\)\setminus\(\mathcal{B}\(M\)\cup\mathcal{O}\(M\)\cup\mathcal{L}\(M\)\).$$
Suppose $$\mathcal{O}_1\(M\), \dots, \mathcal{O}_K\(M\), \mathcal{L}_1\(M\), \dots, \mathcal{L}_L\(M\), \mathcal{B}\(M\)$$ are disjoint sets. To visualize these conceptions, we provide a diagram in Figure~\ref{fig:dig}.
\begin{figure*}[htbp]
  \centering
  \label{fig:dig}
  \includegraphics[width=0.97\textwidth, keepaspectratio]{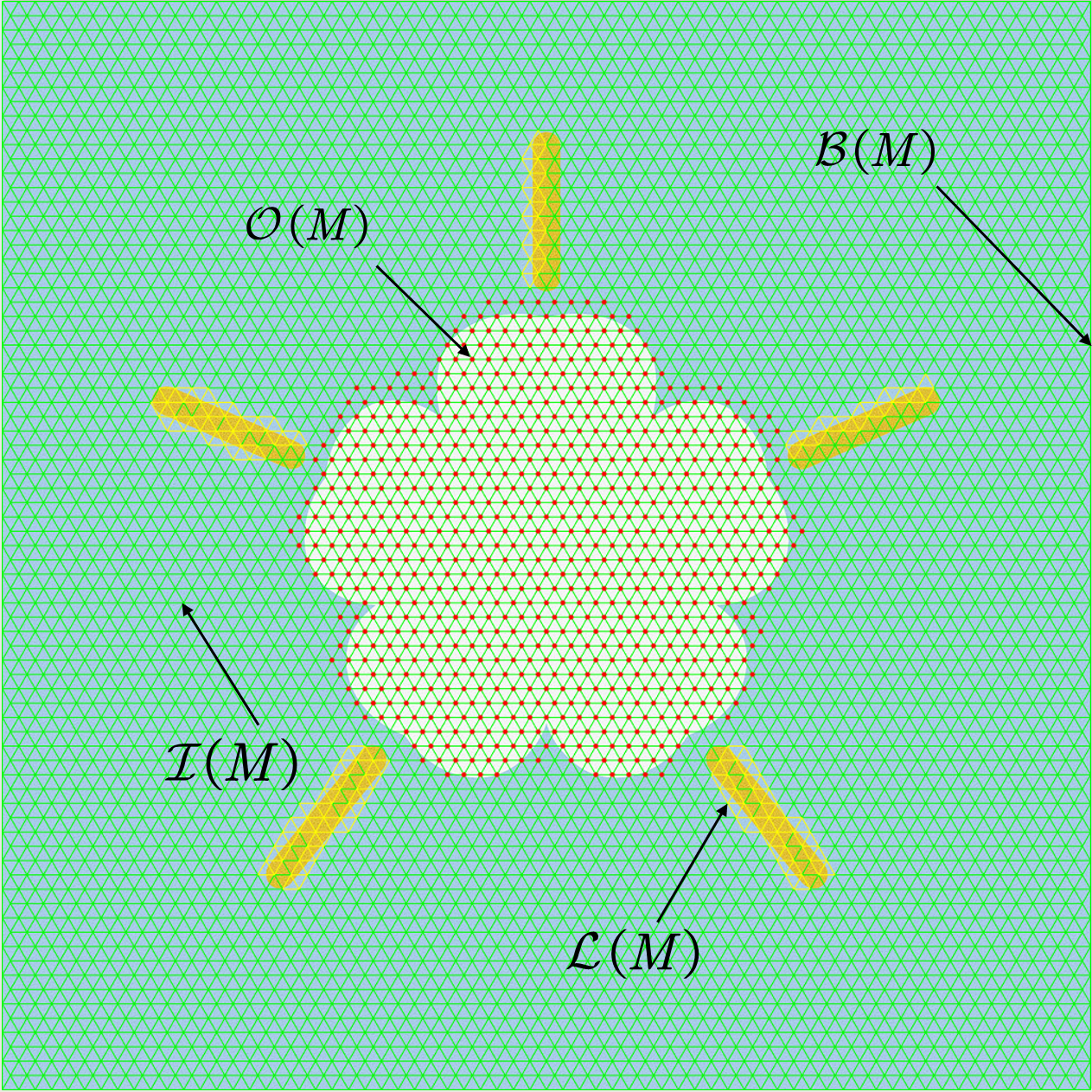}
\caption{A diagram to show the conceptions defined on mesh $M$.}

\end{figure*}
Therefore, we have the ROI, line structure, and boundary constraints on mesh:
\begin{equation}
  \label{equ:con}
  \begin{aligned}
    f \(\mathbf{v}_\ell \) &= r_{\mathcal{O}} \mathbf{v}_\ell  + \mathbf{t}_{\mathcal{O}_i}, \mathbf{v}_\ell \in \mathcal{O}_i\(M\)\\
    f \(\mathbf{v}_\ell \) &= R_{\mathcal{L}_j} \mathbf{v}_\ell  + \mathbf{t}_{\mathcal{L}_j}, \mathbf{v}_\ell \in \mathcal{L}_j\(M\)\\
    f \(\mathbf{v}_\ell \) &= g \(\mathbf{v}_\ell\) , \mathbf{v}_\ell \in \mathcal{B}\(M\)\\
  \end{aligned}
\end{equation}

Let the simplicial mapping $\Pi_Mf: \mathcal{D}_1 \rightarrow \mathcal{D}_2$ be the affine interpolant of $f$ on mesh $M$:
  \begin{equation}
    \begin{aligned}
      \Pi_M f\( \mathbf{v}^{\alpha}_{\ell} \) = \sum\limits_{p=i,j,k}{\alpha _p f\left( \mathbf{v}_{p_{\ell}}\right)} = \sum\limits_{p=i,j,k}{\alpha _p \(u_{p_\ell}, v_{p_\ell}\)}, \\
      \mathbf{v}^{\alpha}_{\ell} \in \tau_\ell \in \mathcal{F}\(M\), \(u_{p_\ell}, v_{p_\ell}\) = f\left( \mathbf{v}_{p_{\ell}}\right) \in \mathcal{D}_2.
    \end{aligned}
\end{equation}
  In other words, for the simplex ${\tau_\ell} = \left[ \mathbf{v}_{i_\ell}, \mathbf{v}_{j_\ell}, \mathbf{v}_{k_\ell} \right]$, 
  $\Pi_Mf\left(\mathbf{v}_{p_{\ell}}\right) = f\(\mathbf{v}_{p_{\ell}}\)$, where $p = i, j, k$
  and $\left.\Pi_M f\right|_{\tau_\ell} (x, y)=\left(\begin{array}{l}\partial_xu\mid_{\tau_\ell}  x + \partial_yu\mid_{\tau_\ell}  y + t^x_{\tau_\ell}  \\ \partial_xv\mid_{\tau_\ell}  x + \partial_yv\mid_{\tau_\ell}  y+t^y_{\tau_\ell} \end{array}\right)$, 
  in which $t^x_{\tau_\ell}, t^y_{\tau_\ell} \in \mathbb{R}$ are translation coefficients of $\tau_{\ell}$ and
  \begin{equation}
    \begin{aligned}
    \left( \begin{matrix}
      \partial_xu\mid_{\tau_\ell}&		\partial_xv\mid_{\tau_\ell}\\
      \partial_yu\mid_{\tau_\ell}&		\partial_yv\mid_{\tau_\ell}\\
    \end{matrix} \right) :=&\,\,\nabla \Pi _{M}f\mid_{\tau_\ell} = \left( \begin{matrix}
      \sum\limits_{p=i,j,k}{A_{p}^{\tau_\ell}u_{p_\ell}}&		\sum\limits_{p=i,j,k}{A_{p}^{\tau_\ell}v_{p_\ell}}\\
      \sum\limits_{p=i,j,k}{B_{p}^{\tau_\ell}u_{p_\ell}}&		\sum\limits_{p=i,j,k}{B_{p}^{\tau_\ell}v_{p_\ell}}\\
    \end{matrix} \right)
  \end{aligned}
  \end{equation}
  is the first derivative of $\Pi _{M}f$, where
  \begin{equation}
    \begin{aligned}
      &A_i^{\tau_\ell}=\left(y_j-y_k\right) / 2 \operatorname{Area}\({\tau_\ell}\),A_j^{\tau_\ell}=\left(y_k-y_i\right) / 2 \operatorname{Area}\({\tau_\ell}\),
      A_k^{\tau_\ell}=\left(y_i-y_j\right) / 2 \operatorname{Area}\({\tau_\ell}\); \\
      &B_i^{\tau_\ell}=\left(x_k-x_j\right) / 2 \operatorname{Area}\({\tau_\ell}\),
      B_j^{\tau_\ell}=\left(x_i-x_k\right) / 2 \operatorname{Area}\({\tau_\ell}\),  B_k^{\tau_\ell}=\left(x_j-x_i\right) / 2 \operatorname{Area}\({\tau_\ell}\).
    \end{aligned}
  \end{equation}

Discrete conformal energy on mesh $M$ is defined as 
\begin{equation}
  \begin{aligned}
    E^C_{M}(f) &= E^C\(\Pi_{M}f\) = -\operatorname{Area}(\mathcal{D}_2) + \tfrac{1}{2}\sum\limits_{\tau _{\ell}\subseteq \mathcal{D}_1}{\| \nabla \Pi _{M}f|_{{\tau_\ell}} \|^2\operatorname{Area} \left( \tau_{\ell} \right)}\ge 0.
  \end{aligned}
\end{equation}

\subsection{Minimizer in the discrete setting}
In the discrete setting, we want to get the optimal simplicial mapping on mesh $M$.
Denote the discrete feasible region as $\mathcal{S}_{M} = \{f\in H^1\(\mathcal{D}_1\) \mid f = \Pi_{M}f;\,\,
  f \(\mathbf{v}_\ell \) = r_{\mathcal{O}} \mathbf{v}_\ell  + \mathbf{t}_{\mathcal{O}_i},\, \mathbf{v}_\ell \in \mathcal{O}_i\(M\), \, i = 1, 2, \dots, k;\,\,
  f \(\mathbf{v}_\ell \) = R_{\mathcal{L}_j} \mathbf{v}_\ell  + \mathbf{t}_{\mathcal{L}_j},\, \mathbf{v}_\ell \in \mathcal{L}_j\(M\),\, j = 1, 2, \dots, l;\,\,
  f \(\mathbf{v}_\ell \) = g \(\mathbf{v}_\ell\) ,\, \mathbf{v}_\ell \in \mathcal{B}\(M\) \}$.
  Denote that $n_v = \operatorname{Card} \left(\mathcal{V}\left(M\right)\right)$, and $n_i = \operatorname{Card} \left(\mathcal{I}\left(M\right)\right)$.
The simplicial mapping $f\in\mathcal{S}_{M}$ can be written in matrix form: 
$\mathbf{f} = \left( f\left( \mathbf{v}_1 \right)^{\top}, f\left( \mathbf{v}_2 \right)^{\top}, \dots, f\left( \mathbf{v}_{n_{i}} \right)^{\top}, \dots, f\left( \mathbf{v}_{n_{v}} \right)^{\top} \right)^{\top}$,
in which $ f\( \mathbf{v}_i\) = \(u_i, v_i\)$, $i = 1, 2, \dots, n_v$; $\mathbf{v}_1, \mathbf{v}_2, \dots, \mathbf{v}_{n_i} \in \mathcal{I}\(M\)$ and $\mathbf{v}_{n_i+1}, \dots, \mathbf{v}_{n_v} \in \mathcal{V}\(M\) \setminus \mathcal{I}\(M\)$.

In fact, discrete conformal energy $E^C_{M}(f)$ is a quadratic polynomial with respect to $\{\(u_{\ell},v_{\ell}\) = f\(\mathbf{v}_{\ell}\)\mid \mathbf{v}_{\ell}\in \mathcal{I}\(M\)\}$,
which means we can solve a system of linear equations to find the optimized simplicial mapping.
This ensures that the computational complexity of the algorithm is acceptable for real-time computing.

The linear equations are listed below:
\begin{equation}
  \begin{aligned}
  &\left\{ \begin{array}{r}
    0 = \tfrac{\partial E_{M}^{C}(\mathbf{f})}{\partial u_i}=\sum\limits_{\tau \in \mathcal{N} _i}{\operatorname{Area} \left( \tau \right) \(A_{i}^{\tau}\partial _xu\mid _{\tau}+B_{i}^{\tau}\partial _yu\mid _{\tau}\)}\\
    0 = \tfrac{\partial E_{M}^{C}(\mathbf{f})}{\partial v_i}=\sum\limits_{\tau \in \mathcal{N} _i}{\operatorname{Area} \left( \tau \right) \(A_{i}^{\tau}\partial _xv\mid _{\tau}+B_{i}^{\tau}\partial _yv\mid _{\tau}\)}\\
  \end{array} \right., \\
  &\text{for } \mathbf{v}_i = \(u_i, v_i\) \in \mathcal{I}\(M\), i = 1, 2, \dots, n_i,
\end{aligned}
\end{equation}
where $\mathcal{N}_i = \{\tau \in \mathcal{F}\(M\)\mid \mathbf{v}_i\in\tau\}$.

Then, the equations can be presented as follows:
\begin{equation}
  L \mathbf{f} = \mathbf{O},
\end{equation}
where $L$ is a $n_i \times n_v$ matrix, whose entries satisfy
\begin{equation}
  \left(L \right)_{i, j} = \left\{ \begin{aligned} 
    \sum\limits_{\tau \in \mathcal{N}_i \text{ and } \left[i, j \right]\subset \tau }{\operatorname{Area} \left( \tau \right) \left( A_{i}^{\tau}A_{j}^{\tau}+B_{i}^{\tau}B_{j}^{\tau} \right)},  \quad &\text{if }\left[i, j\right] \in \mathcal{E}, j \neq i; \\ 
     \sum\limits_{\tau \in \mathcal{N}_i}{\operatorname{Area} \left( \tau \right) \left( A_{i}^{\tau} A_{i}^{\tau} + B_{i}^{\tau} B_{i}^{\tau} \right)}, \quad &\text { if } j=i; \\
     0, \quad &\text { otherwise; }\end{aligned}\right.
\end{equation}
and $\mathbf{O}$ is a $n_i \times 2$ zero matrix.

Denote that $\mathbf{f} = \(\mathbf{f}_a^{\top}, \mathbf{f}_b^{\top}\)^\top$, where $\mathbf{f}_a = \(f\left( \mathbf{v}_1 \right)^{\top}, f\left( \mathbf{v}_2 \right)^{\top}, \dots, f\left( \mathbf{v}_{n_i} \right)^{\top}\)^{\top}$
and $\mathbf{f}_b = \(f\left( \mathbf{v}_{n_i+1} \right)^{\top}, f\left( \mathbf{v}_{n_i+2} \right)^{\top}, \dots, f\left( \mathbf{v}_{n_v} \right)^{\top}\)^{\top}$.
Correspondingly, the matrix $L$ can be written in the form of a block matrix:
$
  L = \left[\begin{matrix}
	L^a&		L^b\\
\end{matrix} \right]
$,
where $L^a$ is the first $n_i$ columns of $L$,
so we have
\begin{equation}
\label{eq:linear}
  L^a \mathbf{f}_a = -  L^b \mathbf{f}_b.
\end{equation}
If we can prove $L^a$ is inevitable, we can ensure the uniqueness and existence of the minimizer in the discrete setting, which will be talked about later. Denote the minimizer in the discrete setting as $f^*_{M} \in \mathcal{S}_{M}$.

\subsection{Bijection correction algorithm}
In some cases, particularly with complex constraints or extreme ratios, the optimal simplicial mapping $f_{M}$ may not be a bijective function due to low mesh accuracy or extreme constraints.
To ensure the warping map is bijective and orientation-preserving, we introduce the bijection correction of simplicial mapping, 
which will relax the ROI and line constraints to ensure the bijection at the expense of the shape of ROIs and line structures.

We initialize the constraints relaxation set as $\mathcal{N}^0 = \{\mathbf{v} \in \mathcal{V}(M)\cap \tau \mid \det \left( \nabla f_M \mid _{\tau} \right) < 0\}$ 
and the iterative step is $\mathcal{N}^{k+1} = \{\mathbf{v} \in \mathcal{V}(M)\mid \left[ \mathbf{v},\, \mathbf{w}\right]\in \mathcal{E}(M), \mathbf{w}\in \mathcal{N}^{k}\}$. 
If the output of the conformal energy minimizer $f^*_M$ is not orientation-preserving, we will start our iterative bijection correction algorithm. 
The corrected function $\phi^k_M$ is generated by equations:
\begin{equation}
  \label{equ:bi}
  \left\{ 
    \begin{aligned}
      \phi^k_M \(\mathbf{v}_\ell \) = r_{\mathcal{O}} \mathbf{v}_\ell  + \mathbf{t}_{\mathcal{O}_i},&\,   \text{ for } \mathbf{v}_\ell \in \mathcal{O}_i\(M\)\setminus\mathcal{N}^k\\
      \phi^k_M \(\mathbf{v}_\ell \) = R_{\mathcal{L}_j} \mathbf{v}_\ell  + \mathbf{t}_{\mathcal{L}_j},&\,  \text{ for } \mathbf{v}_\ell \in \mathcal{L}_j\(M\)\setminus\mathcal{N}^k\\
      \phi^k_M \(\mathbf{v}_\ell \) = g \(\mathbf{v}_\ell\) ,&\, \text{ for } \mathbf{v}_\ell \in  \mathcal{B}\(M\)\\
      \mathbf{0} = \tfrac{\partial E^C_M}{\partial \phi^k_M (\mathbf{v}_\ell)},&\, \text{ otherwise}.
    \end{aligned}\right.
\end{equation}
Then, we will check whether $\phi^k_M$ is orientation-preserving. If so, we will set it as the final output; otherwise, we will go to another step of the iterative method.

\subsection{ROIs, line structures and boundary condition}
ROIs and line structures can be either manually marked or automatically detected. 
For ROI detection, the saliency map is a powerful tool \cite{itti1998model, harel2006graph, ren2009image}. We can set a threshold to segment regions with saliency values exceeding it as ROIs. \cite{mukhopadhyay2015survey}. Using a more advanced network (e.g., \cite{ravi2024sam}) for ROI detection gives better results.
Line structures can be realized by Hough transform \cite{mukhopadhyay2015survey} or other line-detecting neural network (e.g., \cite{zhou2019end}).

The parameters in the model $r_{\mathcal{O}}, R_{\mathcal{L}_j}, \mathbf{t}_{\mathcal{O}_i}, \mathbf{t}_{\mathcal{L}_j}$ and boundary condition $g$ can be fixed in advance or calculated by some algorithm.
Here, we give an initialization method to determine these parameters.

We can use the least square method to find the constrained map, which is the closest one to the discrete harmonic map. 
In other words, we can initialize these parameters by computing the approximate solution to the overdetermined system: 

\begin{equation}
  \label{equ:roi}
  \begin{aligned}
  &\left\{ \begin{array}{l}
    \mathbf{0} = \tfrac{\partial E_{M}^{C}(\mathbf{f})}{\partial f\(\mathbf{v}_\ell\) }, \,\text{ for } \mathbf{v}_\ell \in \mathcal{V}\( M\)\setminus \mathcal{B}\(M\);\\
    0 = \tfrac{\partial E_{M}^{C}(\mathbf{f})}{\partial f\(\mathbf{v}_\ell\) } \cdot \mathbf{n}\(f\(\mathbf{v_\ell}\)\), \,\text{for } \mathbf{v}_\ell \in \mathcal{B}\(M\)\setminus \{\(0, 0\),  \(a, 0\), \(0, b\), \(a, b\)\};\\
    f\(0, 0\) = \(0, 0\), \,f\(a, 0\) = \(wa, 0\), \,f\(0, b\) = \(0, b\),\, f\(a, b\) = \(wa, b\);\\
  \end{array} \right.
  \end{aligned}
\end{equation}
where $\mathbf{n} \(f \( \mathbf{v}_\ell \)\)$ is the normal vector of $\partial \mathcal{D}_2$ at $f\(\mathbf{v}_\ell\)$, and
$$
\left\{
\begin{aligned}
  f \(\mathbf{v}_\ell\) &= r_{\mathcal{O}} \mathbf{v}_\ell + \mathbf{t}_{\mathcal{O}_i},\, \text{ for }  \mathbf{v}_\ell\in \mathcal{O}_i\(M\),\\
  f \(\mathbf{v}_\ell\) &= R_{\mathcal{L}_j} \mathbf{v}_\ell + \mathbf{t}_{\mathcal{L}_j}, \,\text{ for } \mathbf{v}_\ell\in \mathcal{L}_j\(M\).\\
\end{aligned}
\right.
$$
then, we can set the boundary condition as $g \(\mathbf{v}_\ell \) = f\(\mathbf{v}_\ell\), \mathbf{v}_\ell \in \mathcal{B}\(M\)$.

\subsection{The main algorithm}
The body of our bijective retargeting algorithm is as follows: 
Firstly, we use simplicial mapping to approximate the conformal energy minimizer in the continuous setting. 
Next, we need to do the bijection correction to the simplicial mapping if it is not orientation-preserving. 
Then we can get the bijective warping map we want. 
In the case the parameters $r_{\mathcal{O}}, \mathbf{t}_{\mathcal{O}_i}, R_{\mathcal{L}_j}, \mathbf{t}_{\mathcal{L}_j}, g$ are not given, we can initialize these parameters by solving (\ref{equ:roi}).
\begin{algorithm}[htbp]  
  \caption{The main algorithm}  
  \label{alg:Main}
  \begin{algorithmic}[1]
    \Require  
    Initial image $I_1$; domain $\mathcal{D}_1$ and codomain $\mathcal{D}_2$; ROIs $\mathcal{O}_i, i = 1, 2, ..., k$; line structures $\mathcal{L}_j, j=1, 2, \dots, l$; parameters $r_{\mathcal{O}}, \mathbf{t}_{\mathcal{O}_i}, R_{\mathcal{L}_j}, \mathbf{t}_{\mathcal{L}_j}$ and bijective boundary condition $g$ (optimal).
    \Ensure  
    Resized image $I_2$ with prescribed ratio.
    \If{$r_{\mathcal{O}}, \mathbf{t}_{\mathcal{O}_i}, R_{\mathcal{L}_j}, \mathbf{t}_{\mathcal{L}_j}, g$ are not given}
      \State Initialize them by solving (\ref{equ:roi});
    \EndIf
    \State Compute the simplicial approximation $\mathbf{f}^*_{M}$ of $f^*$ on mesh $M$;
    \If {the approximation $\mathbf{f}^*_{M}$ is not orientation-preserving}
      \State Compute $\phi_M^{0}$ by solving (\ref{equ:bi});
      \State Set $k = 1$;
      \While{$\phi^{k-1}_M$ is not orientation-preserving}
          \State Compute $\phi^{k}_M$ by solving (\ref{equ:bi});
          \State $k = k + 1$;
      \EndWhile
    \EndIf
    \State Denote the output of bijection correction algorithm as $\phi^*_M$;
    \State Compute $I_2$ by $I_2 = I_1 \circ \(\phi^*_{M}\)^{-1}$;\\
    \Return The image $I_2$.
  \end{algorithmic}
\end{algorithm}

\section{Analysis of Model}
\label{sec:3}
\subsection{Preliminaries}
\subsubsection{Simplicial interpolation in Sobolev space}
\begin{theorem}[Theorem 1 in \protect\cite{van2014approximation}]
  \label{thm:app}
  Let $u \in L^q\left(\mathbb{R}^n\right)$. If $u$ is weakly differentiable and if $D u \in L^p\left(\mathbb{R}^n\right)$, then for every $\varepsilon>0$ there exists a triangulation $M$ of $\mathbb{R}^n$ such that
$$
\int_{\mathbb{R}^n}\left|D\left(u-\Pi_{M} u\right)\right|^p+\int_{\mathbb{R}^n}\left|u-\Pi_{M} u\right|^q \leq \varepsilon.
$$
\end{theorem}
This theorem also applies to vector-valued functions.
It shows that the simplicial interpolation can be a promising approximation for the mappings in $H^1$ space if the mesh is appropriately selected, which gives the theoretical foundation of triangular mesh interpolation of warping maps. In other words, there exists a sequence of meshes $\{M_k\}_{k\in\mathbb{N}}$ such that $\lim_{k\rightarrow \infty}{\Pi_{M_k}{f}} = f$ with respect to $H^1$ norm for any $f\in \mathcal{R}$.

\subsubsection{Subdivision of simplicial complex}
\begin{definition}[Subdivision of simplicial complex, \cite{rourke2012introduction}]
  Simplicial complex $J$ is a subdivision of another simplicial complex $K$, written $J \triangleleft K$, if $|J|=|K|$ and each cell of $J$ is contained in a cell of $K$.
\end{definition}

Subdivision is partitioning simplices of the original complex into smaller simplices. All the vertices of the original mesh are also included in the subdivision mesh. The most commonly used subdivision is the barycentric subdivision.

Denote the original mesh be $M_0 := M$. We can have a sequence of subdivision meshes $\{M_k\}_{k \in \mathbb{N}}$ such that $M_{k+1} \triangleleft M_{k}$, $k \in \mathbb{N}$.

\subsubsection{Assumptions}
In order to analyze the model, we shall impose some assumptions first.
\begin{assumption}
  \label{ass:omega}
$\Omega$ is a bounded set with Lipschitz boundary.
\end{assumption}
\begin{remark}
  This assumption always holds if the boundaries of ROIs are sufficiently regular.
\end{remark}

\begin{assumption}
  \label{ass:mk}
  For $k \in \mathbb{N}^+$, $M_k$ is a simply connected Delaunay mesh.
\end{assumption}

\begin{remark}
This assumption is related to the existence and uniqueness of discrete minimizer on the mesh $M_k$.
\end{remark}

\begin{assumption}
  Suppose for any $k\in\mathbb{N}$, $M_k$ and boundary condition $g$ satisfy the following conditions:\\
  $$
  \left\{
    \begin{aligned}
      \mathcal{O}_i = \left|\sum\limits_{\tau \in \mathcal{O}_i\(\mathcal{F}\(M_k\) \)}{\tau} \right|, \,&i = 1, 2, \dots, k;\\
      \mathcal{L}_j = \left|\sum\limits_{\left[ \mathbf{v}_\ell,\mathbf{v}_m\right]\in \mathcal{L}_j\(\mathcal{E}\(M_k\)\)}{\left[ \mathbf{v}_\ell,\mathbf{v}_m\right]} \right|,  \, &j = 1, 2, \dots, l;\\
      g\(\alpha_1 \mathbf{v}_l + \alpha_2 \mathbf{v}_m \) = \alpha_1 g\(\mathbf{v}_l\) + \alpha_2 g\(\mathbf{v}_m\),\, &\left[\mathbf{v}_l,\mathbf{v}_m\right] \in \mathcal{E} \(M_k\) \cap \partial \mathcal{D}_1, \alpha_1, \alpha_2\ge0, \alpha_1 + \alpha_2 =1 ;
    \end{aligned}
  \right.
  $$
  where $\mathcal{O}_i\(\mathcal{F}\(M_0\) \) :=\{\tau= \left[\mathbf{v}_l,\mathbf{v}_m,\mathbf{v}_n\right] \mid \mathbf{v}_l, \mathbf{v}_m, \mathbf{v}_n\in \mathcal{O}_i\(M_k\)\}$ 
  and $\mathcal{L}_j\(\mathcal{E}\(M_k\)\) := \{\left[ \mathbf{v}_\ell,\mathbf{v}_m\right] \in\mathcal{E}\(M_k \) \mid \mathbf{v}_\ell, \mathbf{v}_m\in \mathcal{L}_j\(M_k\)\}$.
  \label{ass:roi}
\end{assumption}

\begin{remark}
If this assumption holds, we can have $\Pi_{M_k}f \in \mathcal{S}_{M_k}$, so we obtain an embedding from the set $\mathcal{R}$ to the set $\mathcal{S}_{M_k}$ in the sense of interpolation equivalence classes.
\end{remark}

By Theorem \ref{thm:app}, we can know that for any mappings in $\mathcal{R}$, there exists a simplicial interpolation which is close enough to itself with respect to $H^1$ distance. 
If the mesh $M_0$ and the subdivision method is appropriately selected, the following assumption is reasonable.
\begin{assumption}
  \label{ass}
  The minimizer in the continuous setting $f^*$ satisfies $f^* \in \overline{\mathcal{S}}$, in which $$\mathcal{S} := \bigcup_{k\in\mathbb{N}} \mathcal{S}_k := \bigcup_{k\in\mathbb{N}} \mathcal{S}_{M_k}.$$
\end{assumption}

\begin{assumption}
  \label{ass:co}
  No quadrilateral of $M$ has cocircular vertices.
\end{assumption}

\begin{remark}
    The assumption above means for any two triangles sharing a common edge, the sum of their opposite angles is not equal to $\pi$. All meshes generated by the Delaunay triangulation that contain no right angles satisfy the assumption, including the most widely used isometric grid.
    
\end{remark}

\subsection{Existence and uniqueness of the minimizer}
Since $f\(\mathcal{D}_1\)$, $f\mid_{\mathcal{O}}$, and $f\mid_{\mathcal{L}}$ are fixed, the constrained optimization problem (\ref{op:continuous}) can be translated into another form:
\begin{equation}
  \begin{aligned}
    \min_{f\in H^1\(\Omega\)} &{E^D_\Omega\(f\):=\frac{1}{2} \int_{\Omega}{\|\nabla f\|^2}}\\
    \text{s.t. }&f\mid _{\partial \mathcal{D}_1}=g,\\
    &f\mid _{\partial \mathcal{O}_i}\left( \mathbf{x} \right) =r_{\mathcal{O}}\mathbf{x}+\mathbf{t}_{\mathcal{O}_i},\, i = 1, 2, \dots, k,\\
    &f\mid _{\partial \mathcal{L}_j}\left( \mathbf{x} \right) =R_{\mathcal{L}_j}\mathbf{x}+\mathbf{t}_{\mathcal{L}_j},\, j = 1, 2, \dots, l,
  \end{aligned}
\end{equation}
where $E^D_\Omega$ equals $E^C$ plus a constant.

\begin{theorem}[Alaoglu's Theorem, \cite{conway2019course}]
  \label{thm:banach}
If $X$ is a normed space, then the closed unit ball of $X^*$ is weak-star compact.
\end{theorem}

\begin{theorem}[Rellich-Kondrachov theorem, \cite{evans2010partial}]
  \label{thm:rel}
  Let $\Omega \subset \mathbb{R}^N, N\in \mathbb{N}$ be a bounded Lipschitz domain, the embedding $H^1(\Omega) \hookrightarrow L^2(\Omega)$ is compact and continuous.
\end{theorem}

\begin{theorem}[Existence and uniqueness of the minimizer in the continuous setting]
  \label{thm:con}
If Assumption \ref{ass:omega} holds, there exists one and only one minimizer $f^* \in \mathcal{R}$ of conformal energy $E^C$.
\end{theorem}
\begin{proof}
  In the bounded set $\Omega$ with Lipschitz boundary, because $E^D_\Omega \ge 0$, for any $\epsilon\ge 0$ there exists a sequence of functions $\{u_k\}_{k\in\mathbb{N}}$ such that $E^D_\Omega\(u_k\) \le \epsilon + \inf_{f\in\mathcal{R}}\{E^D_\Omega\(f\)\} $ and $\lim_{k\rightarrow \infty}{E^D_\Omega\(u_k\)}=\inf_{f\in\mathcal{R}}\{E^D_\Omega\(f\)\}$, 
  so there exists $r>0$ such that $\{u_k\}_{k\in\mathbb{N}}\subset B_r^{H^1\(\Omega\)}:=\{u\in H^1\(\Omega\)\mid\left|u\right|_{H^1\(\Omega\)}\le r\}$ because of Poincaré inequality.
  By Theorem \ref{thm:banach}, we can have $\{ u_{k_p} \}_{p\in\mathbb{N}}\subseteq \{u_k\}_{k\in\mathbb{N}}$ which satisfies $w\text{-}\lim_{p\rightarrow \infty}{u_{k_p}} = u$. 
  By Theorem \ref{thm:rel}, $\lim_{p\rightarrow \infty}{u_{k_p}} = u^*$ in the sence of strongly convergence in $L^2\(\Omega\)$. 
  Due to the fact that $\lim_{p\rightarrow \infty}{E\(u_{k_p}\)} = E\(u^*\)$, we can get $\lim_{p\rightarrow \infty}{\nabla u_{k_p}} = \nabla u^*$ with respect to $L^2\(\Omega\)$ norm.
  $|f|_{H^1\(\Omega\)} \le |f|_{L^2\(\Omega\)} + |\nabla f|_{L^2\(\Omega\)}$, so  $\lim_{p\rightarrow \infty}{u_{k_p}} = u^*$ in the sence of strongly convergence in $H^1\(\Omega\)$. 
  $\mathcal{R}$ is a close set in $H^1\(\Omega\)$, so we can have $u^*\in\mathcal{R}$.
  For any $f_1, f_2\in H^1\(\Omega\)$, $|E^D_\Omega\(f_1\) - E^D_\Omega\(f_2\)| \le \tfrac{1}{2}|\nabla f_1 - \nabla f_2|_{L^2\(\Omega\)} \le \tfrac{1}{2}|f_1 - f_2|_{H^1\(\Omega\)}$, so $E^D_\Omega$ is continuous in $H^1\(\Omega\)$. 
  Therefore, $E^D_\Omega\(u^*\) = \lim_{p\rightarrow \infty}{E\(u_{k_p}\)}\le \inf_{f\in\mathcal{R}}\{E^D_\Omega\(f\)\}+\epsilon$. 
  Due to arbitrariness of $\epsilon>0$, we have $E^D_\Omega\(u^*\) = \inf_{f\in\mathcal{R}}\{E^D_\Omega\(f\)\}$. 
  Therefore, we can get $E^D_\Omega\(u^*\) = \inf_{f\in\mathcal{R}}\{E^D_\Omega\(f\)\}$. The existence of the minimizer holds.

  It is obvious that $E^D_\Omega$ is a strongly convex functional, so the uniqueness of the minimizer is also ensured.

  Because $E^D_\Omega$ equals $E^C$ plus a constant, the unique minimizer of $E^D_\Omega$ is also the minimizer of $E^C$.
\end{proof}
Therefore, we can claim the existence and uniqueness of the minimizer in the continuous setting. Denote the minimizer as $f^*$ in $\mathcal{R}$.

Then, we need to ensure the existence and uniqueness of the conformal energy minimizer in the discrete setting. {In order to analyze the properties of the linear system (as in \ref{eq:linear}), we first need to clarify the nature of the matrix $L^a$. In fact, the matrix $L$ derived from the minimization of the discrete conformal energy is the well-known cotangent Laplacian matrix for the mesh $M$ \cite{yueh2017efficient}. The matrix $L^a$ is the principal submatrix of this global Laplacian matrix $L$ that corresponds to the set of interior, unconstrained vertices $\mathcal{I}\left( M \right)$. The following lemma details its specific form and key properties.}
\begin{lemma}
  \label{lemma:<0}
$L^a$ is a principal submatrix of the discrete cotangent Laplacian matrix 
\begin{equation}
  \label{Lap}
  \left(L \right)_{i, j}= \begin{cases}-\frac{1}{2}\left(\cot \alpha_{i j}+\cot \alpha_{j i}\right) & \text { if }[i, j] \text { is an edge, } j \neq i \\ -\sum\limits_{k \neq i} L_{i, k} & \text { if } j=i \\ 0 & \text { otherwise, }\end{cases} 
\end{equation} 
where $\alpha_{i j}$ and $\alpha_{j i}$ are the two angles opposite to the edge $[i, j]$ connecting vertices $i$ and $j$ on the mesh $M$; Crucially, if the mesh $M$ is a Delaunay triangulation, The off-diagonal entries of matrix $L$ are non-positive, i.e. $(L)_{i,j} \le 0, \, i\neq j$.
\end{lemma}
\begin{proof}
  Suppose there are two adjacent simplices $\tau_1 = \left[\mathbf{v}_i, \mathbf{v}_j, \mathbf{v}_k \right], \tau_2 = \left[ \mathbf{v}_i, \mathbf{v}_\ell, \mathbf{v}_j \right] \in \mathcal{F}\(M\)$, then the corresponding entry of edge $\left[i, j \right]$ in $L^a$ is 
  \begin{equation}
    \begin{aligned}
    \(L^a\)_{i, j} = &\sum\limits_{\tau \in \mathcal{N}_i \text{ and } \left[i, j \right]\subset \tau }{\operatorname{Area} \left( \tau \right) \( A_{i}^{\tau}A_{j}^{\tau}+B_{i}^{\tau}B_{j}^{\tau} \)}\\
                   = &\sum\limits_{p = 1,2}{\operatorname{Area} \left( \tau_p \right) \( A_{i}^{\tau_p}A_{j}^{\tau_p}+B_{i}^{\tau_p}B_{j}^{\tau_p} \)}\\
                   = &- \frac{1}{2} \left( \frac{\(\mathbf{v}_j -\mathbf{v}_k\) \cdot \(\mathbf{v}_i -\mathbf{v}_k\)}{2\operatorname{Area}\(\tau_1\)} +\frac{\(\mathbf{v}_j -\mathbf{v}_\ell\) \cdot \(\mathbf{v}_i -\mathbf{v}_\ell\)}{2\operatorname{Area}\(\tau_2\)}\right)\\
                   = &-\tfrac{1}{2}\(\cot\(\alpha_1\) +\cot\(\alpha_2\)\) \\
                   = &\left(L \right)_{i, j}.
    \end{aligned}
  \end{equation}
  For off-diagonal entries of matrix $L$, the values are either $0$ or $-\tfrac{1}{2}\left( \cot \alpha_{i j}+\cot \alpha_{j i} \right)$, where $\alpha_{i j}$ and $\alpha_{j i}$ are the two angles opposite to the edge $[i, j]$. Because $M$ is a Delaunay rectangular mesh, we have $\alpha_1 + \alpha_2 \le \pi$, which means $-\left( \cot \alpha_{i j} + \cot \alpha_{j i}\right) \le 0$. Therefore, $(L)_{i,j} \le 0, \, i\neq j$.
\end{proof}

\begin{definition}
  (i) A matrix $A \in \mathbb{R}^{m \times n}$ is said to be nonnegative(positive) if all entries of $A$ are nonnegative(positive).\\
  (ii) A matrix $A \in \mathbb{R}^{n \times n}$ is said to be an M-matrix if $A=s I-B$, where $B$ is nonnegative and $s \geq \rho(B)$.\\
  {
  (iii) A matrix $A \in \mathbb{R}^{n \times n}$ is a symmetric positive definite matrix with nonpositive off-diagonal entries}
\end{definition}

\begin{lemma}[\cite{berman1994nonnegative}]
\label{lemma:sm}
{
  A symmetric nonsingular M-matrix is a Stieltjes matrix.
  }
\end{lemma}

\begin{lemma}[Theorem 1.4.10 in \cite{molitierno2016applications}]
  \label{lemma:submatrix}
  Suppose $A \in \mathbb{R}^{n \times n}$ is a singular, irreducible M-matrix. Then, each principal submatrix of $A$ other than $A$ itself is a nonsingular $M$-matrix.
\end{lemma}

\begin{lemma}[\cite{yueh2017efficient}]
  \label{lemma:L}
  The $n_v \times n_v$ Laplacian matrix $L$ defined in (\ref{Lap}), is a singular M-matrix.
\end{lemma}

\begin{theorem}[Existence and uniqueness of the minimizer in the discrete setting]
  \label{thm:diss}
  $L^a$ is a Stieltjes matrix, and there is one and only one minimizer $f^*_M$ of $E^C_{M}\(\mathbf{f}\)$, whose matrix form is:
  \begin{equation}
    \mathbf{f}^*_{M} = \left[ \begin{matrix}-\(L^a\)^{-1} L^b \mathbf{f}_b\\ \mathbf{f}_b\end{matrix}\right].
  \end{equation}
\end{theorem}
\begin{proof}
  From Lemma \ref{lemma:L}, we know that $L$ is a singular M-matrix. 
  Because the mesh $M$ is connected, the adjacent matrix $A$ of graph $M$ is irreducible.
  Due to Lemma \ref{lemma:<0}, the matrix $L^a$ can be represented as $L = D + E$, 
  where $D$ is a diagonal matrix and $E$ is a negative matrix whose diagonal entries are all 0.
  Since $A$ and $E$ has the following relationship:
  \begin{equation}
    A_{i,j} =  \begin{cases} 
      0, & E_{i,j} = 0, \\ 
      1, & E_{i,j} \ne 0,\end{cases} 
  \end{equation}
  $E$ is is irreducible too, so $L$ is a irreducible M-matrix. 
  $L^a$ is a principal submatrix of $L$, so $L^a$ is a nonsingular M-matrix by Lemma \ref{lemma:submatrix}. Additionally, since $L$ is a symmetric matrix, 
  the principal submatrix $L^a$ is also symmetric. By Lemma \ref{lemma:sm}, we get that $L^a$ is a Stieltjes matrix, which is positive definite and invertible. Therefore, there is one and only one minimizer $f^*_M$ of $E^C_{M}\(\mathbf{f}\)$, whose matrix form is:
  \begin{equation}
    \mathbf{f}^*_M := \left[ \begin{matrix}-\(L^a\)^{-1} L^b \mathbf{f}_b\\ \mathbf{f}_b\end{matrix}\right].
  \end{equation}

For computational complexity of the discrete minimizer, denote $\operatorname{nnz}\left(A\right)$ as the number of nonzero elements in matrix $A$ and $\kappa$ is the average degree of vertices in the mesh $M$. The computational complexity of $\mathbf{h} = L^b \mathbf{f}_b$ is $O\left(\operatorname{nnz}\left(L^b\right) \right) = O\left(\kappa n_v\right)$. 
The primary computational cost is in solving the sparse linear system $L^a \mathbf{f_a} = \mathbf{h}$ using the Nested Dissection algorithm \cite{gilbert1986analysis}. This algorithm recursively partitions the graph using separators $S$ of size $|S| = \mathcal{O}(\sqrt{n_i})$. The complexity is derived from recurrence relations dominated by factoring the dense separator block at each level. The operation count follows the recurrence $T(n_i) = 2 T(n_i/2) + \mathcal{O}(|S|^3) = 2 T(n_i/2) + \mathcal{O}(n_i^{3/2})$, which solves to a total of $\mathcal{O}(n_i^{3/2})$ \cite{cormen2022introduction}. As the factorization cost is asymptotically dominant over the pre-processing and solve phases, the total complexity is $\mathcal{O}(n_i^{3/2})$.

\end{proof}
\label{sec:complexity}

\subsection{Convergence of sequence of minimizers}
If Assumption \ref{ass:mk} is true, the existence and uniqueness of $f^*_k := f^*_{M_k}$ is ensured.
In this subsection, we will talk about the convergence property of the sequence of minimizers $\{f^*_k\}_{k\in\mathbb{N}}$.

\begin{lemma}[Representation of $\mathcal{S}$ with $\mathcal{R}$]
  \label{lemma:rep}
  If Assumption \ref{ass:mk} and Assumption \ref{ass:roi} hold, $\mathcal{S}_k = \{f\mid f = \Pi_{M_k} \phi,\, \phi \in \mathcal{R} \}\subseteq \mathcal{R}$.
\end{lemma}
\begin{proof}
  If Assumption \ref{ass:roi} holds, for any $\phi \in\mathcal{R}\subset H^1\(\mathcal{D}_1\)$, we have 
  \begin{equation}
    \left\{
    \begin{aligned}
      &\Pi_{M_k}\phi = \Pi_{M_k}\Pi_{M_k}\phi;\\
      &\(\Pi_{M_k}\phi\) \(\mathbf{v}_\ell \) = r_{\mathcal{O}} \mathbf{v}_\ell  + \mathbf{t}_{\mathcal{O}_i},\, \quad \text{for } \mathbf{v}_\ell \in \mathcal{O}_i\(M\)\subseteq \mathcal{O}_i, \, i = 1, 2, \dots, k;\\
      &\(\Pi_{M_k}\phi\) \(\mathbf{v}_\ell \) = R_{\mathcal{L}_j} \mathbf{v}_\ell  + \mathbf{t}_{\mathcal{L}_j},\, \quad \text{for } \mathbf{v}_\ell \in \mathcal{L}_j\(M\)\subseteq \mathcal{L}_j,\, j = 1, 2, \dots, l;\\
      &\(\Pi_{M_k}\phi\) \(\mathbf{v}_\ell \) = g \(\mathbf{v}_\ell\) ,\,\text{for } \mathbf{v}_\ell \in \mathcal{B}\(M\).
    \end{aligned}
    \right.
  \end{equation}
  Therefore, $\{f\mid f = \Pi_{M_k} \phi,\, \phi \in \mathcal{R}  \}\subseteq \mathcal{S}_k$.

  As for any simplicial mapping $f\in\mathcal{S}_k \subset H^1\(\mathcal{D}_1\)$, for $\mathbf{x} = \alpha_i\mathbf{v}_i + \alpha_j\mathbf{v}_j + \alpha_k\mathbf{v}_k \in \left[\mathbf{v}_i, \mathbf{v}_j, \mathbf{v}_\ell\right] \subseteq \mathcal{O}_i$,
  \begin{equation}
    \begin{aligned}
    f\(\mathbf{x}\) &= r_{\mathcal{O}} \( \alpha_i\mathbf{v}_i + \alpha_j\mathbf{v}_j + \alpha_k\mathbf{v}_\ell\) + \( \alpha_i + \alpha_j + \alpha_\ell\)\mathbf{t}_{\mathcal{O}_i} =  r_{\mathcal{O}} \mathbf{x}\ + \mathbf{t}_{\mathcal{O}_i};
    \end{aligned}
  \end{equation}
  For $\mathbf{x} = \alpha_i\mathbf{v}_i + \alpha_j\mathbf{v}_j \in \left[\mathbf{v}_i, \mathbf{v}_j\right] \subseteq \mathcal{L}_j$,
  \begin{equation}
    \begin{aligned}
    f\(\mathbf{x}\) &= R_{\mathcal{L}_j} \( \alpha_i\mathbf{v}_i + \alpha_j\mathbf{v}_j\) + \( \alpha_i + \alpha_j \)\mathbf{t}_{\mathcal{L}_j}
    =  R_{\mathcal{L}_j} \mathbf{x}\ + \mathbf{t}_{\mathcal{L}_j};
    \end{aligned}
  \end{equation}
  For $\mathbf{x} = \alpha_i\mathbf{v}_i + \alpha_j\mathbf{v}_j \in \left[\mathbf{v}_i, \mathbf{v}_j\right] \subseteq \partial \mathcal{D}_1$,
  \begin{equation}
    \begin{aligned}
    f\(\mathbf{x}\) &=  \alpha_i g \(\mathbf{v}_i\) + \alpha_j g\(\mathbf{v}_j\)
    = g\(\alpha_i\mathbf{v}_i + \alpha_j\mathbf{v}_j\)
    = g\(\mathbf{x}\). 
    \end{aligned}
  \end{equation}
  Therefore, we have $f\in \mathcal{R}$. 
  Noticing that $f = \Pi_{M_k} f $, we can get $\mathcal{S}_k = \{f\mid f = \Pi_{M_k} \phi,\, \phi \in \mathcal{R}\}\subseteq \mathcal{R}$.
\end{proof}

\begin{lemma}[Increasing sequence of sets]
  \label{lemma:increasing}
  If Assumption \ref{ass:mk} and Assumption \ref{ass:roi} hold, $\{\mathcal{S}_k\}_{k\in\mathbb{N}}$ is an increasing sequence of sets.
  Additionally, if Assumption \ref{ass} is true, the minimizer in the continuous setting $f^* \in \overline{\lim_{k\rightarrow \infty} {\mathcal{S}_k}}$.
\end{lemma}
\begin{proof}
  From the proof of Lemma \ref{lemma:rep}, we know that the ROIs, line structures, boundary conditions and $f\(\mathcal{D}_1\)=\mathcal{D}_2$ constraints are all equivalent.
  We also have $\Pi_{M_{k+1}}\Pi_{M_{k}}f = \Pi_{M_{k}}f$ for $f\in\mathcal{R},\,k\in\mathbb{N}$, so $\Pi_{M_{k}}f\in\mathcal{S}_{k+1}$.
  Therefore, $\mathcal{S}_{1}\subseteq\mathcal{S}_{2}\subseteq\mathcal{S}_{3}\subseteq\dots$.

  Given the increasing sequence of sets $\{\mathcal{S}_k\subseteq\mathcal{R}\}_{k\in\mathbb{N}}$, we have $\lim_{k\rightarrow \infty}\mathcal{S}_k=\bigcup_k^\infty{\mathcal{S}_k} = \mathcal{S}$. 
  By Assumption \ref{ass}, $f^* \in \overline{\mathcal{S}}$ holds, so $f^* \in \overline{\lim_{k\rightarrow \infty} {\mathcal{S}_k}}$.
\end{proof}

\begin{theorem}[Convergence of minimums and minimizers]
\label{thm:conv}
  If Assumption \ref{ass:omega}, Assumption \ref{ass:mk}, Assumption \ref{ass:roi}, and Assumption \ref{ass} are true, we have
  \newline
  \begin{enumerate}[label=(\roman*)]
    \item $E^C\(f^*_k\) \rightarrow E^C\(f^*\),\, k \rightarrow \infty$;
    \item there exists a convergent subsequence $\{f^*_{k_p}\}_{p\in\mathbb{N}}$ of $\{f^*_{k}\}_{k\in\mathbb{N}}$;
    \item for any convergent subsequence $\{f^*_{k_q}\}_{q\in\mathbb{N}}$, $\lim_{q\rightarrow \infty}\left|f^*-f^*_{k_q}\right|_{H^1\(\mathcal{D}_1\)} = 0$.
  \end{enumerate}
\end{theorem}

\begin{proof}
  Denote that $E^C_k\(f\) := E^C_{M_k}\(f\)$. By Lemma \ref{lemma:rep}, we have $\mathcal{S}_k = \{f\mid f=\Pi_{M_{k}}\phi, \phi\in\mathcal{R}\},\, k\in\mathbb{N}$, 
  so 
  $$
  \begin{aligned}
    \min_{f\in \mathcal{S}_k}{E^C_k}\(f\) = \min_{f\in \mathcal{R}}{E^C_k}\(f\) = \min_{f\in \mathcal{S}_k}{E^C}\(f\).
  \end{aligned}
  $$ 
  From Lemma \ref{lemma:increasing}, we know $\mathcal{S}_1\subseteq \mathcal{S}_2\subseteq \mathcal{S}_3\subseteq\dots\subseteq\mathcal{R}$, so 
  $$
  \begin{aligned}
    \min_{f\in \mathcal{S}_k}{E^C_k}\(f\) &= \min_{f\in \mathcal{S}_k}{E^C}\(f\) \le \min_{f\in \mathcal{S}_{k+1}}{E^C}\(f\)\\
    &=\min_{f\in \mathcal{S}_{k+1}}{E^C_{k+1}}\(f\),\,k\in\mathbb{N}.
  \end{aligned}
  $$
  Due to the existence and uniqueness in the continuous setting (Theorem \ref{thm:con}) and the discrete setting (Theorem \ref{thm:diss}), $E^C_k\(f^*_k\) = \min_{f\in \mathcal{S}_k}{E^C_k}\(f\)$ and $E^C\(f^*\) = \min_{f\in \mathcal{R}}{E^C}\(f\)$, so $\{E^C_k\(f^*_k\) \}_{k\in\mathbb{N}}$ is a decreasing sequence with lower bound $E^C\(f^*\)$. 
  By Lemma \ref{lemma:increasing}, there exists a sequence of mappings $\{\psi_k \in \mathcal{S}_k\}_{k\in\mathbb{N}}\subset \mathcal{R}$ such that $\lim_{k\rightarrow \infty} \psi_k = f^*$. 
  Since the functional $E^C$ is continuous, $\lim_{k\rightarrow \infty}E^C\(\psi_k\) = E^C\(f^*\)$. 
  It is obvious that $E^C\(\psi_k\) \ge E^C\(f^*_k\)$. 
  Therefore, we can get $\lim_{k\rightarrow \infty}E^C\(f^*_k\) = E^C\(f^*\)$ by dominated convergence theorem.

  Because $\lim_{k\rightarrow \infty}E^C\(f^*_k\) = E^C\(f^*\)$, 
  for any $\epsilon>0$, there exists $N \in \mathbb{N}^+$ such that $E^C\(f^*_k\) \le E^C\(f^*\)  + \epsilon$ for any $k\ge N\in\mathbb{N}^+$. 
  $\operatorname{int}\(\mathcal{D}_1\)\subset \mathbb{R}^2$ is bounded open set with Lipschitz boundary, and the measure of $\partial \mathcal{D}_1$ is 0.
  There exists $r>0$, such that the sequence of functions $\{f^*_{N + k}\}_{k\in\mathbb{N}}\subset B_r^{H^1\(\mathcal{D}_1\)}:=\{ f\in H^1\(\mathcal{D}_1\)\mid\left|f \right|_{H^1}<r \}$ due to Poincaré inequality. 
  Then, by Theorem \ref{thm:banach}, there exists a subsequence $\{f^*_{N + k_p}\}_{p\in\mathbb{N}}$ of $\{f^*_{N + k}\}_{k\in\mathbb{N}}$ such that $w\text{-}\lim_{p\rightarrow \infty}{f^*_{N + k_p}} = \phi^*$. 
  From Theorem \ref{thm:rel}, $\lim_{p\rightarrow \infty}{f^*_{N + k_p}} = \phi^*$ with respect to $L^2\(\mathcal{D}_1\)$ norm. 
  We have $\lim_{p\rightarrow \infty}{E^C\(f^*_{N + k_p}\)} = E^C \(\phi^*\)$. 
  In other words, $\lim_{p\rightarrow \infty}{Df^*_{N + k_p}} = D\phi^*$ with respect to $L^2\(\mathcal{D}_1\)$ norm. 
  Since $\left|f\right|_{H^1} = \sqrt{\left|f\right|^2_{L^2\(\mathcal{D}_1\)} + \left|Df\right|^2_{L^2\(\mathcal{D}_1\)}} \le \left|f\right|_{L^2\(\mathcal{D}_1\)} + \left|Df\right|_{L^2\(\mathcal{D}_1\)}$, we have $\lim_{p\rightarrow \infty}{f^*_{N + k_p}} = \phi^*$ in $H^1$.

  Because $\mathcal{R}$ is a close set in $H^1$, $\phi^*\in \mathcal{R}$. 
  By uniqueness of the minimizer of $E^C$ in $\mathcal{R}$, we have $\phi^*=f^*$. 
  Therefore, there exists a subsequence $\{f^*_{N + k_p}\}_{p\in\mathbb{N}}$ of $\{f^*_{k}\}_{k\in\mathbb{N}}$ such that $\lim_{p\rightarrow \infty}f^*_{N + k_p} = f^*$.
\end{proof}

\subsection{Analysis of bijection correction algorithm}
We have three questions about the bijection correction algorithm: Can we ensure the output map is a bijection? Does the algorithm converge? If so, what is the maximum iterative number of the algorithm? 
We will explore these three questions in this subsection.

In our algorithm, the output simplicial mapping of the conformal energy minimizer is guaranteed to be non-degenerate if the retargeting ratio $w \neq 0$.
Lipman has given a theorem in \cite{lipman2014bijective}, which ensures the bijection of non-degenerate orientation preserving simplicial mapping. 
\begin{theorem}[Theorem 1 in \cite{lipman2014bijective}]
  \label{theorem:ori}
  A non-degenerate orientation preserving simplicial mapping $f: M \rightarrow \mathbb{R}^d$ of a $d$-dimensional compact mesh with boundary $M$, $f: M \rightarrow M^\prime$ is a bijection if the boundary map $\left.f\right|_{\partial M}: \partial M \rightarrow \partial M^\prime$ is bijective.
\end{theorem}
Because the output mapping is orientation-preserving, we can ensure the output is a bijection by this theorem.

Then, we will show the algorithm's convergence and maximum iterative number.
\begin{definition}[Convex combination map]
Let $f: D \rightarrow \mathbb{R}^2$ be a simplicial mapping whose mesh is $M$, and suppose, for every interior vertex $\mathbf{v} \in \mathcal{V}(M)\setminus \mathcal{B}(M)$, that there exist weights $\lambda_{\mathbf{v} \mathbf{w}}>0$, for $\mathbf{w} \in \mathcal{N}_\mathbf{v}^1$, such that
$$
\sum\limits_{\mathbf{w} \in \mathcal{N}_\mathbf{v}^1} \lambda_{\mathbf{v} \mathbf{w}}=1
$$
and
$$
f(\mathbf{v})=\sum\limits_{\mathbf{w} \in \mathcal{N}_\mathbf{v}^1} \lambda_{\mathbf{v} \mathbf{w}} f(\mathbf{w}),
$$
in which $\mathcal{N}_\mathbf{v}^1 = \{\mathbf{w}\in \mathcal{V}(M) \mid \left[ \mathbf{v}, \mathbf{w} \right]\in \mathcal{E}\(M\)\}$.
We will call $f$ a convex combination map.
\end{definition}

\begin{proposition}[Proposition 3.2 in \cite{floater2003one}]
  \label{pro:co}
The discrete conformal energy minimizer $f_{M}$ with only boundary constraints on mesh $M$ is a convex combination map if and only if $M$ is a finite Delaunay triangulation and no quadrilateral of $M$ has cocircular vertices.
\end{proposition}

\begin{theorem}[Theorem 6.1 in \cite{floater2003one}]
  \label{thm:one}
  Suppose $\mathcal{S}$ is any triangulation and that $f: D \rightarrow \mathbb{R}^2$ is a convex combination mapping which maps $\partial D$ homeomorphically into the boundary $\partial \Omega$ of some convex region $\Omega \subset \mathbb{R}^2$. 
  Then, $\phi$ is one-to-one if and only if no dividing edge $[v, w]$ of $\mathcal{S}$ is mapped by $f$ into $\partial \Omega$, 
  where a dividing edge $\left[\mathbf{v}, \mathbf{w}\right] \in \mathcal{E}(\mathcal{S})$ is an interior edge and both $\mathbf{v}$ and $\mathbf{w}$ are boundary vertices.
\end{theorem}

\begin{corollary}
\label{cor}
  If Assumption \ref{ass:roi} and Assumption \ref{ass:co} hold, the iterative bijection correction algorithm will converge in $n_c$ steps, where $n_c := \min \{n \mid \mathcal{O}(M) \cup \mathcal{L}(M) \setminus \( \bigcup_{k=1}^n \mathcal{N}^k \) = \varnothing\} + 1 <n_v$.
\end{corollary}
\begin{proof}
  In the $n_c$-th iteration, $\mathcal{O}(M)\cup\mathcal{L}(M)\setminus \mathcal{N}^{n_c-1}= \varnothing$, so we only have boundary constraints for $\phi^{n_c}_M$. 
  If Assumption \ref{ass:roi} holds, $\phi^{n_c}_{M}$ maps $\partial \mathcal{D}_1$ homeomorphically into the boundary $\partial \mathcal{D}_2$ of the convex region $\mathcal{D}_2$. 
  If Assumption \ref{ass:co} holds, we can get $\phi^{n_c}_{M}$ is a convex combination map by Proposition \ref{pro:co}. 
  Additionally, $M$ is simply connected mapping. 
  Therefore, for any point $\mathbf{v} \in \operatorname{int}(\mathcal{D}_1)$, there exist $\lambda_{\mathbf{v}\mathbf{w}} \in (0,1)$, such that $\phi^{n_c}_M(\mathbf{v}) = \sum\limits_{\mathbf{w}\in\mathcal{B}(M)} {\lambda_{\mathbf{v}\mathbf{w}} \phi^{n_c}_{M(\mathbf{w})}}$ and $\sum\limits_{\mathbf{w}}{\lambda_{\mathbf{v}\mathbf{w}}} = 1$. 
  Denote that $\phi^{n_c}_{M}(\mathbf{v}) = (u_{\mathbf{v}},v_{\mathbf{v}})$ and $\phi^{n_c}_{M}(\mathbf{w}) = (u_{\mathbf{w}},v_{\mathbf{w}})$, so $(u_{\mathbf{v}},v_{\mathbf{v}}) = \sum\limits_{\mathbf{w}\in\mathcal{B}(M)} {\lambda_{\mathbf{v}\mathbf{w}} (u_{\mathbf{w}},v_{\mathbf{w}})}$. 
  We can get $\min \{u_\mathbf{w}  \mid\mathbf{w} \in \mathcal{B}(M) \} < u_\mathbf{v} < \max \{ u_\mathbf{w}  \mid\mathbf{w} \in \mathcal{B}(M)\}$ and $\min \{v_\mathbf{w}  \mid \mathbf{w} \in \mathcal{B}(M) \} < v_\mathbf{v} < \max \{ v_\mathbf{w}  \mid\mathbf{w} \in \mathcal{B}(M)\}$, 
  so we have $f_M\(\operatorname{int}(\mathcal{D}_1)\)\subseteq f_M\(\operatorname{int}(\mathcal{D}_2)\)$. 
  We can ensure there is no dividing edge being mapped into $\partial \mathcal{D}_2$. 
  By Theorem $\ref{thm:one}$, we can get $\phi^{n_c}_{M}$ is a one-to-one mapping, 
  so $\phi^{n_c}_{M}$ is orientation-preserving and the algorithm converges.
\end{proof}

\section{Experimental results}
\label{sec:4}
\begin{table}[htbp]

\centering
\begin{tabular}{l|cc|cc}
\hline
\hline
\textbf{Ratio} & \multicolumn{2}{c|}{\textbf{Avg. Conformal Energy ($E^C_M$)}} & \multicolumn{2}{c}{\textbf{Avg. Flips (\textperthousand)}} \\
\cline{2-5}
               & \textbf{Our Method} & \textbf{BR \cite{lau2018image}} & \textbf{Our Method} & \textbf{BR \cite{lau2018image}} \\
\hline
0.50           & \textbf{184,568.60} & 240,671.72 & \textbf{0.00} & 21.08 \\
0.75           & \textbf{51,321.07}  & 52,503.28  & \textbf{0.00} & 9.97  \\
\hline
\hline
\end{tabular}
\caption{Quantitative comparison between our method and the BR method on the RetargetMe dataset (87 images). Our method achieves lower distortion energy and, crucially, guarantees bijectivity (zero flips).}
\label{tab}
\end{table}
We have tested the proposed algorithm on various images in the $RetargetMe$ Dataset \cite{rubinstein2010comparative}.
The results show that our model can preserve ROIs and line structures in images while changing the local geometry of the object as little as possible. {Similar to the Beltrami representation (BR) method \cite{lau2018image}, the core of our algorithm involves solving a sparse linear system, ensuring high computational efficiency. To provide a more comprehensive quantitative evaluation, we compare our algorithm with the BR method across all 87 images in the dataset for different aspect ratios. We measure two key metrics: the average conformal energy of the warping map($E^C_M$) on mesh $M$, where lower values indicate less geometric distortion, and the rate of flipped triangles (Flips ‰), which measures the failure of bijectivity. Zero flips indicate a valid, fold-over-free mapping. As summarized in Table~\ref{tab}, our method consistently achieves lower conformal energy. Crucially, it produces zero flipped triangles across all test cases, empirically verifying its bijective nature. In contrast, the BR method's failure to guarantee bijectivity results in a significant number of flipped triangles, particularly under strong compression (21.08‰ at a 0.5 ratio), which is the direct cause of visual artifacts.}

These quantitative findings are visually substantiated by qualitative examples. he non-bijective nature of the BR method leads to geometric instabilities, resulting in fold-overs and content distortion. This is obvious in the results for a skiing image (Fig.~\ref{fig:ski}) and a lotus image (Fig.~\ref{fig:lotus}). For the skiing image, the BR result in Fig.~\ref{fig:ski}(c) exhibits unnatural distortions in the cloud structures. For the lotus image, the BR result in Fig.~\ref{fig:lotus}(c) not only truncates the right side of the lotus but also shows fold-overs near the boundary. In contrast, our method produces smooth, bijective mappings in both cases (Fig.~\ref{fig:ski}(d) and Fig.~\ref{fig:lotus}(d)), preserving the geometric integrity and completeness of the main subjects without introducing artifacts.

\begin{figure*}[htbp]
  \centering
  \begin{subfigure}[htbp]{0.45\textwidth}
    \includegraphics[width=0.97\textwidth, keepaspectratio]{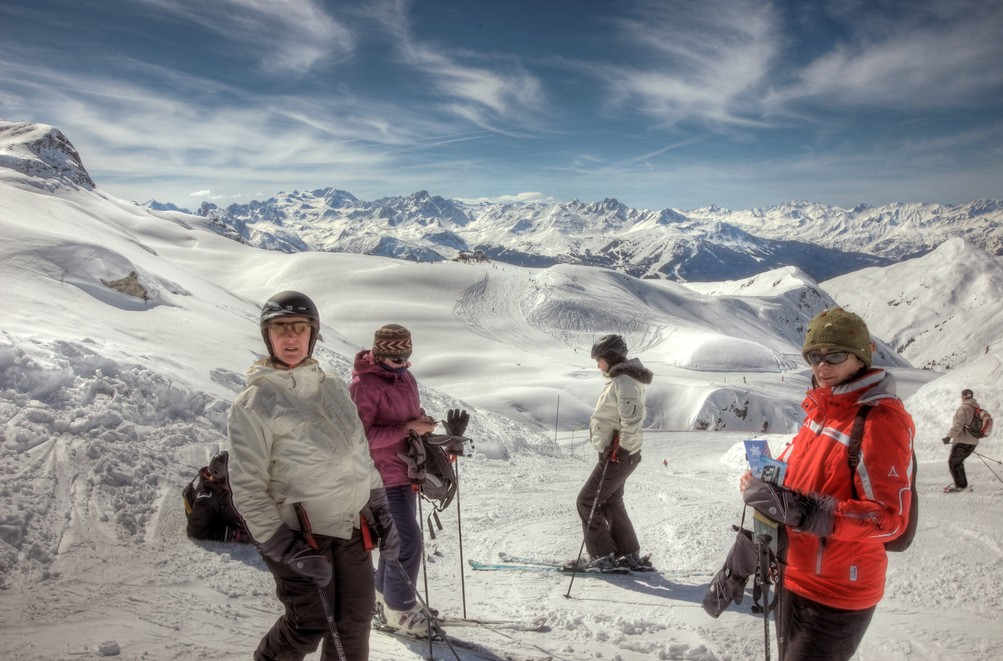}
    \caption{Original image}
  \end{subfigure}%
  \begin{subfigure}[htbp]{0.45\textwidth}
    \includegraphics[width=0.97\textwidth, keepaspectratio]{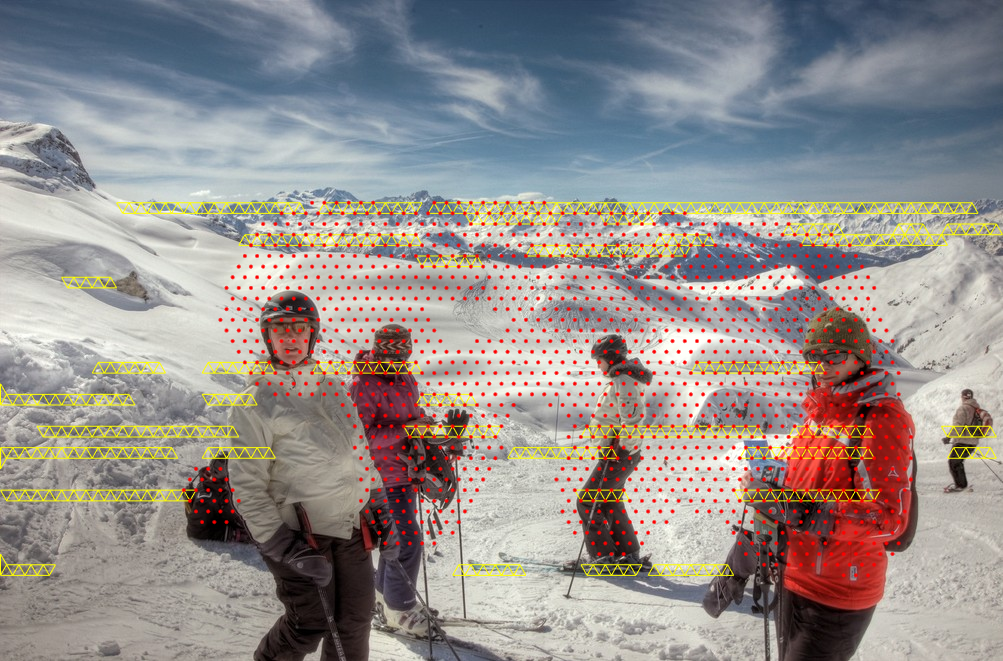}
    \caption{ROIs (red) and line structures (yellow)}
  \end{subfigure}%
    \\
  \begin{subfigure}[htbp]{0.45\textwidth}
    \includegraphics[width=0.97\textwidth, keepaspectratio]{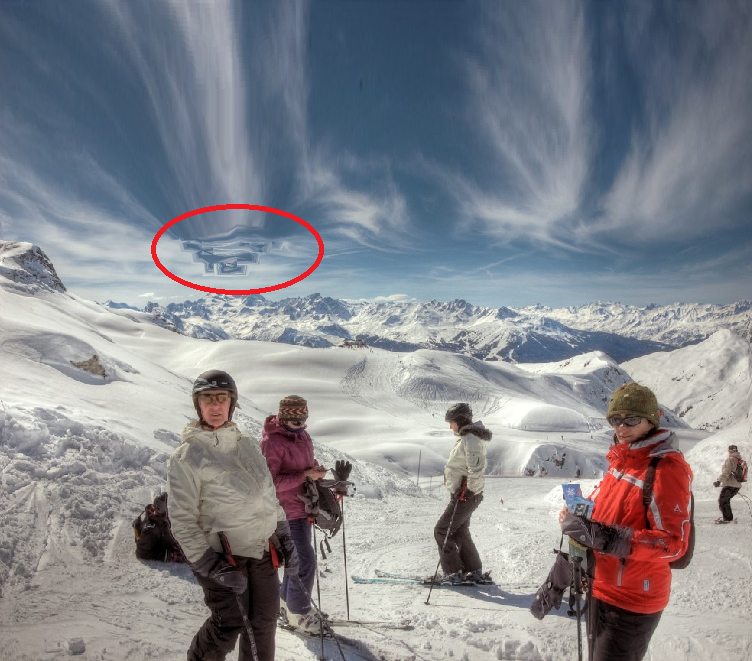}
    \caption{result of BR\cite{lau2018image}}
  \end{subfigure}%
  \begin{subfigure}[htbp]{0.45\textwidth}
    \includegraphics[width=0.97\textwidth, keepaspectratio]{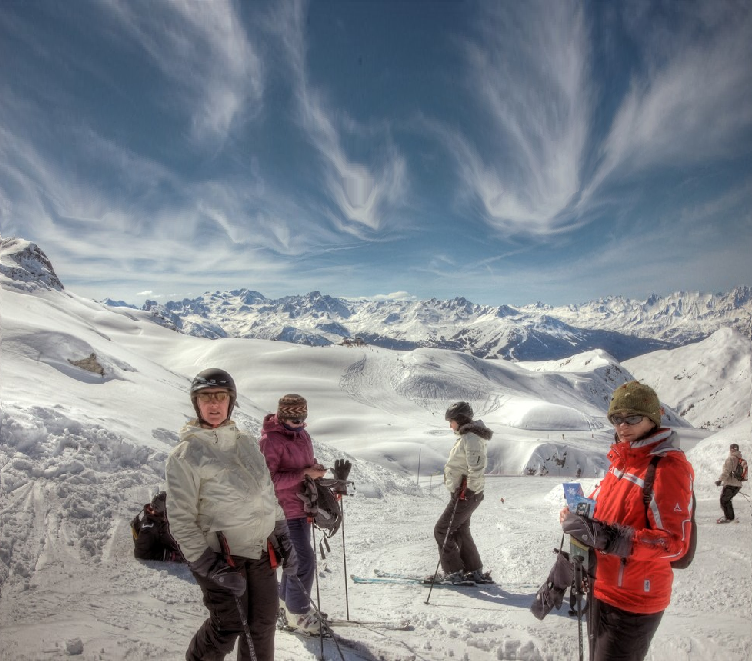}
    \caption{result of our model}
  \end{subfigure}%
\caption{Comparison with BR and our algorithm to resize a skiing image to 75\% of the original width.}
\label{fig:ski}
\end{figure*}
\begin{figure*}[htbp]
\label{fig:lotus}
  \centering
  \begin{subfigure}[htbp]{0.45\textwidth}
    \includegraphics[width=0.97\textwidth, keepaspectratio]{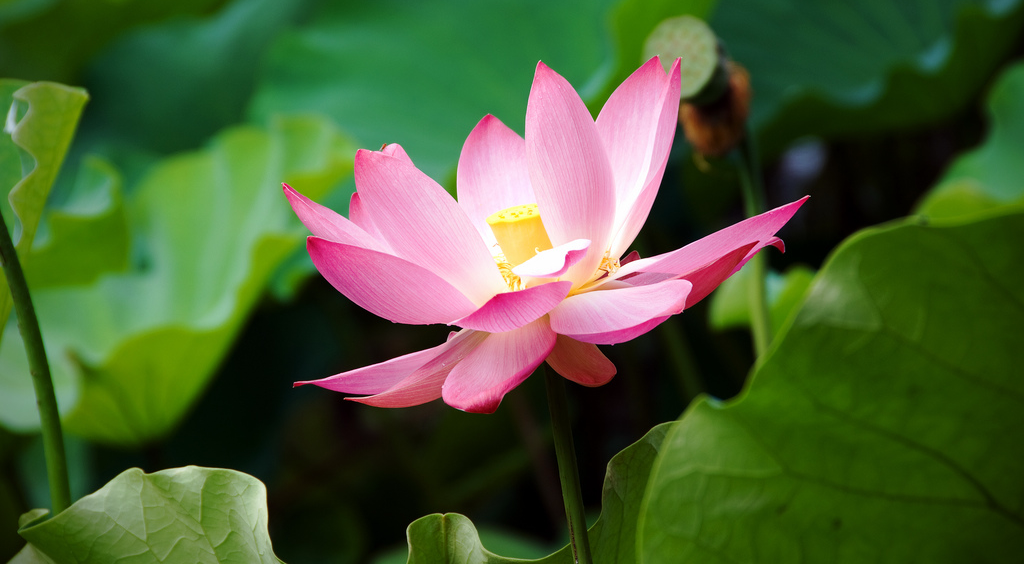}
    \caption{Original image}
  \end{subfigure}%
  \begin{subfigure}[htbp]{0.45\textwidth}
    \label{fig:girl_roi}
    \includegraphics[width=0.97\textwidth, keepaspectratio]{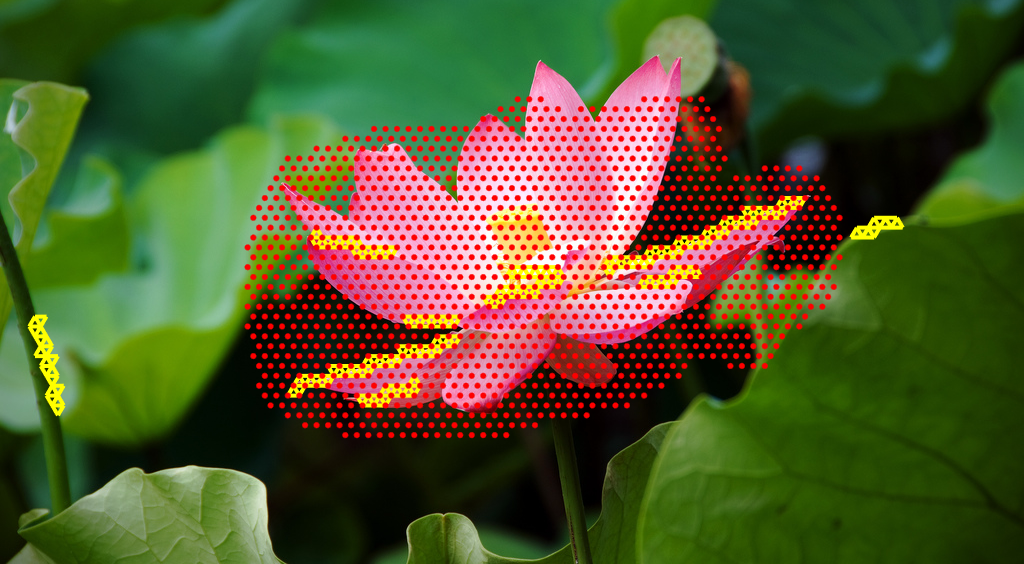}
    \caption{ROIs (red) and line structures (yellow)}
  \end{subfigure}%
    \\
  \begin{subfigure}[htbp]{0.45\textwidth}
    \label{fig:girl_roi}
    \includegraphics[width=0.97\textwidth, keepaspectratio]{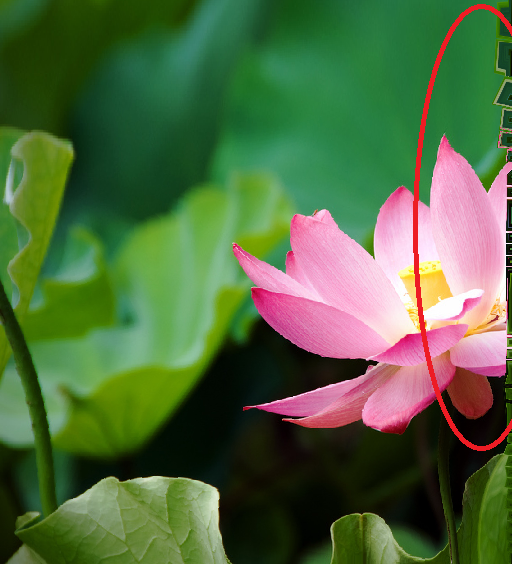}
    \caption{result of BR\cite{lau2018image}}
  \end{subfigure}%
  \begin{subfigure}[htbp]{0.45\textwidth}
    \includegraphics[width=0.97\textwidth, keepaspectratio]{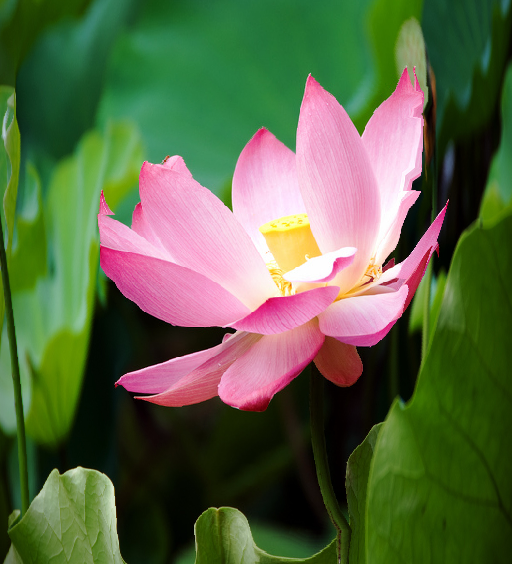}
    \caption{result of our model}
  \end{subfigure}%
\caption{Comparison with BR and our algorithm to resize a lotus image to 50\% of the original width.}
\label{fig:lotus}
\end{figure*}

We find that the performance of our algorithm depends on the proper selection of ROIs and line structures. 
The reflection in our boat picture \ref{fig:boat} looks unnatural because they are not selected as ROIs. 
The results also show the outstanding capacity for retargeting of our model in case of proper selection of ROIs and line structures. A good ROI and line structure detector is 
\begin{figure*}[htbp]
  \centering
  \begin{subfigure}[htbp]{0.66\textwidth}
      \includegraphics[width=0.97\textwidth, keepaspectratio]{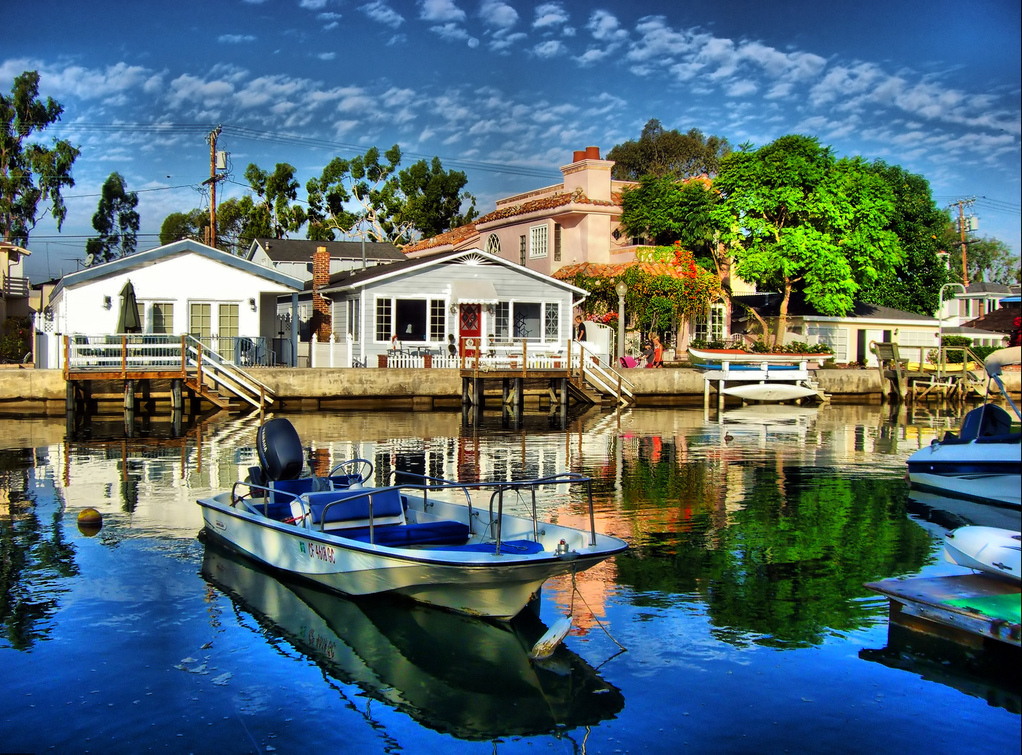}
      \caption{OI}
  \end{subfigure}%
  \begin{subfigure}[htbp]{0.33\textwidth}
      \includegraphics[width=0.97\textwidth, keepaspectratio]{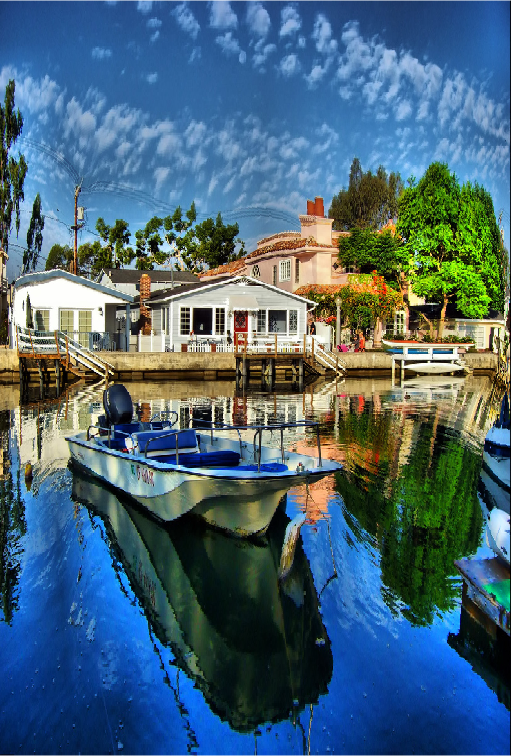}
      \caption{Our}
  \end{subfigure}%

\caption{Resizing a Canalhouse image to 50\% of the original width.}
\label{fig:boat}
\end{figure*}

\begin{figure*}[htbp]
  \centering
  \captionsetup[subfigure]{labelformat=empty}
  \begin{subfigure}[htbp]{0.33\textwidth}
      \includegraphics[width=0.97\textwidth, keepaspectratio]{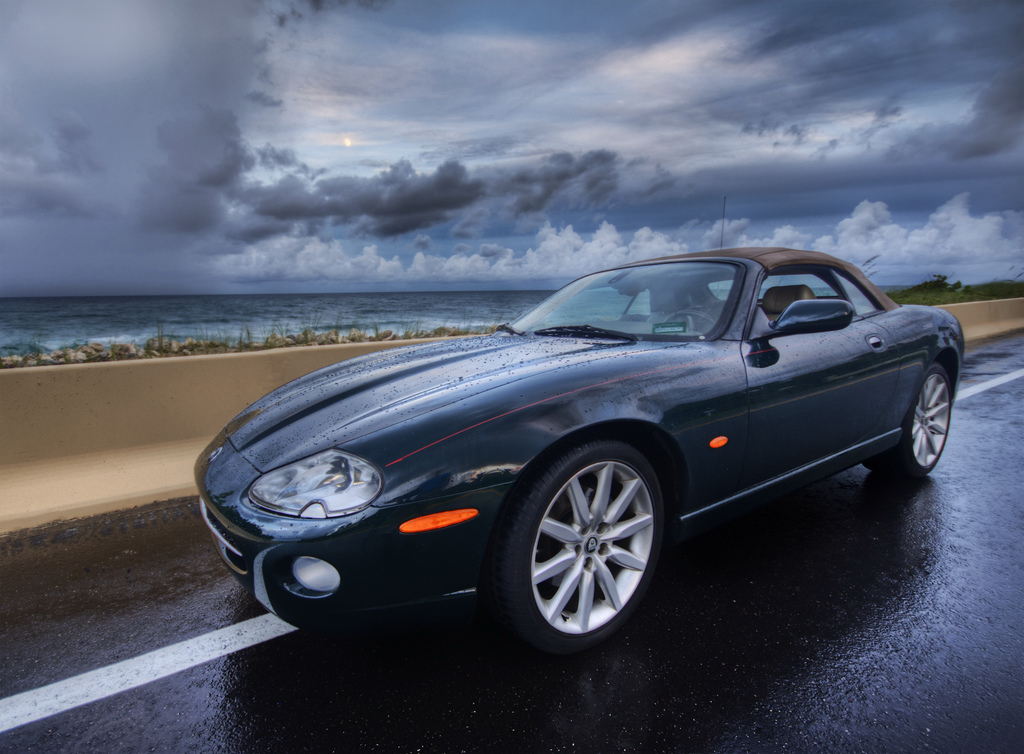}
      \caption{OI}
  \end{subfigure}%
  \begin{subfigure}[htbp]{0.33\textwidth}
      \includegraphics[width=0.97\textwidth, keepaspectratio]{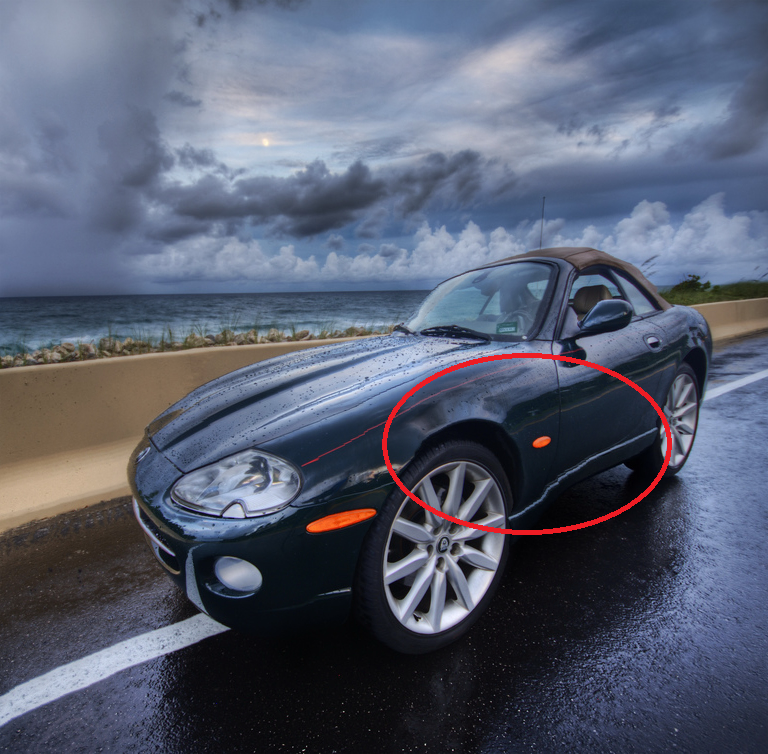}
      \caption{SC\cite{rubinstein2008improved}}
  \end{subfigure}%
  \begin{subfigure}[htbp]{0.33\textwidth}
      \includegraphics[width=0.97\textwidth, keepaspectratio]{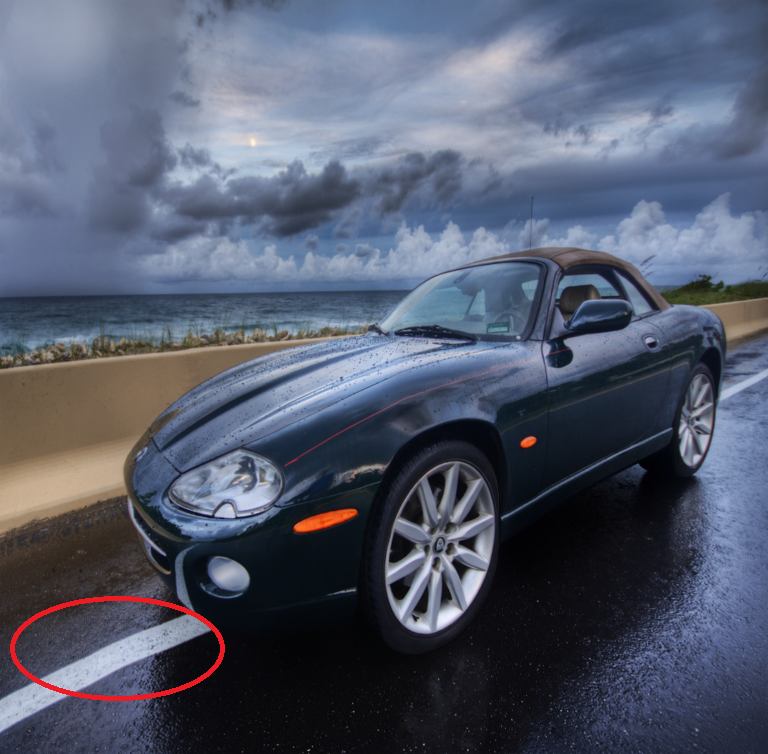}
      \caption{Warp\cite{wolf2007nonhomogeneous}}
  \end{subfigure}%
  \\
  \begin{subfigure}[htbp]{0.33\textwidth}
      \includegraphics[width=0.97\textwidth, keepaspectratio]{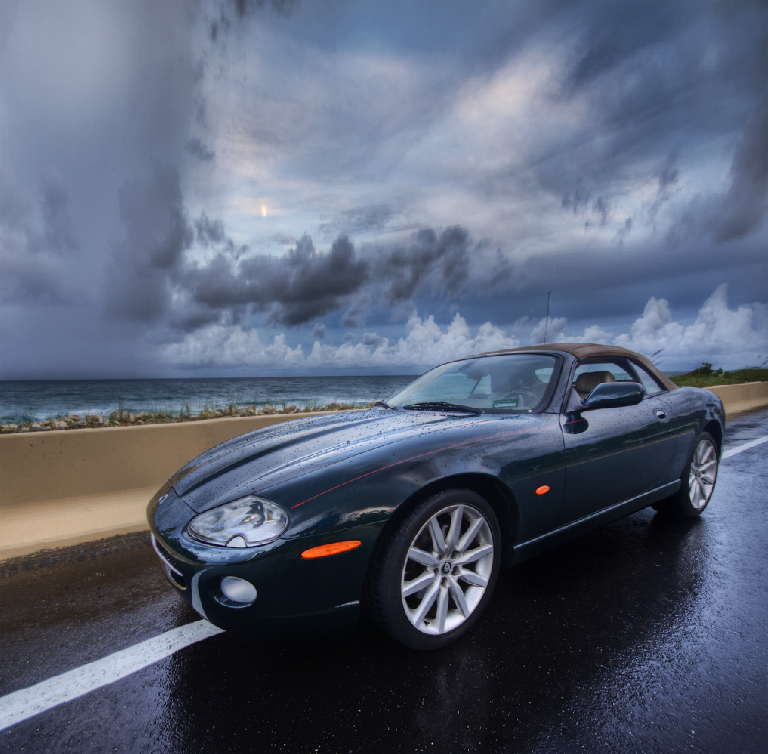}
      \caption{BR\cite{lau2018image}}
  \end{subfigure}%
  \begin{subfigure}[htbp]{0.33\textwidth}
    \includegraphics[width=0.97\textwidth, keepaspectratio]{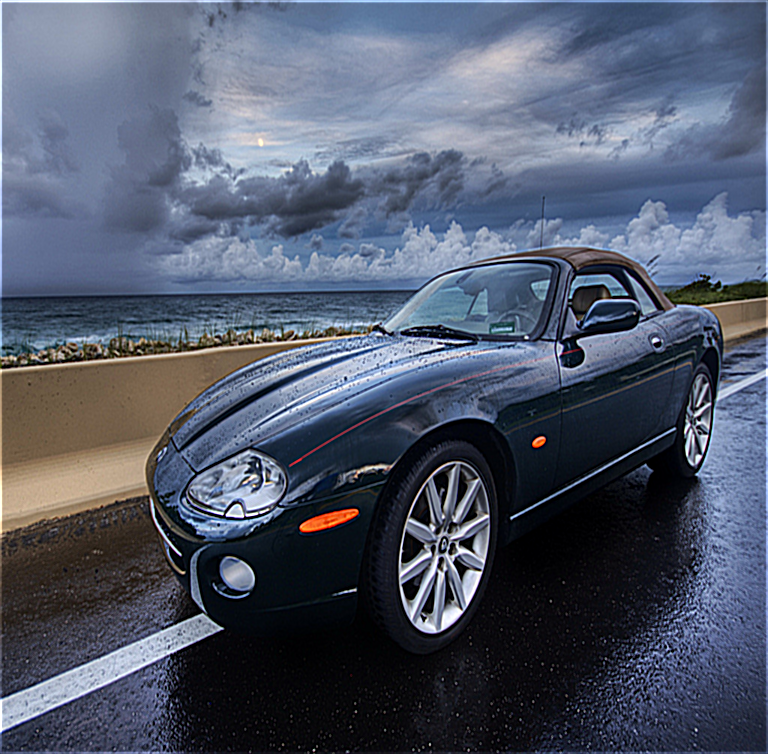}
    \caption{ML\cite{tu2023muller}}
\end{subfigure}%
  \begin{subfigure}[htbp]{0.33\textwidth}
      \includegraphics[width=0.97\textwidth, keepaspectratio]{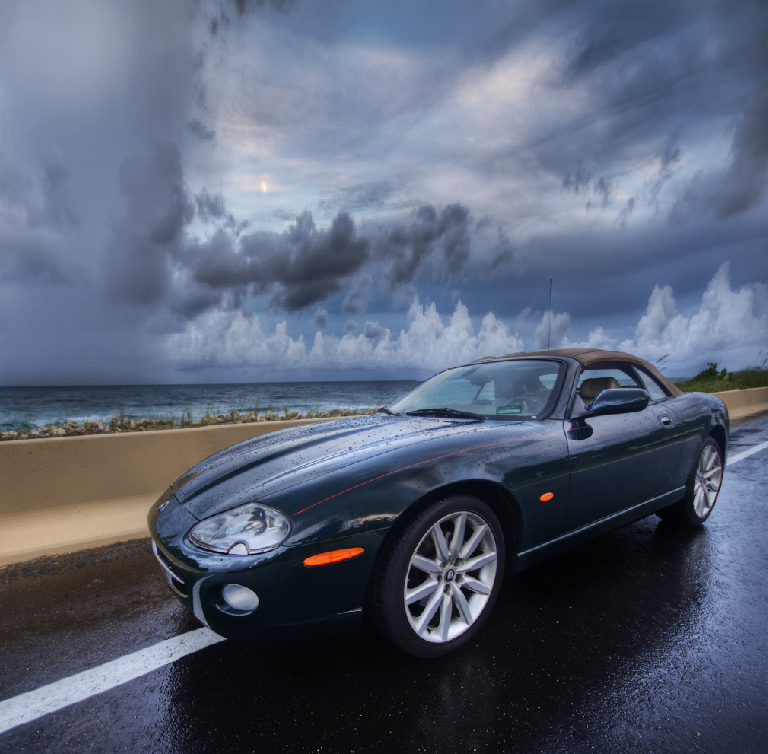}
      \caption{Our}
  \end{subfigure}%
\caption{Comparison of different methods to resize a car image to 75\% of the original width.}
\label{fig:example1}
\end{figure*}

\begin{figure*}[htbp]
    \centering
    \makebox[\textwidth][c]{(a) House}\\
    \begin{subfigure}[htbp]{0.12\textwidth}
        \includegraphics[width=0.97\textwidth, keepaspectratio]{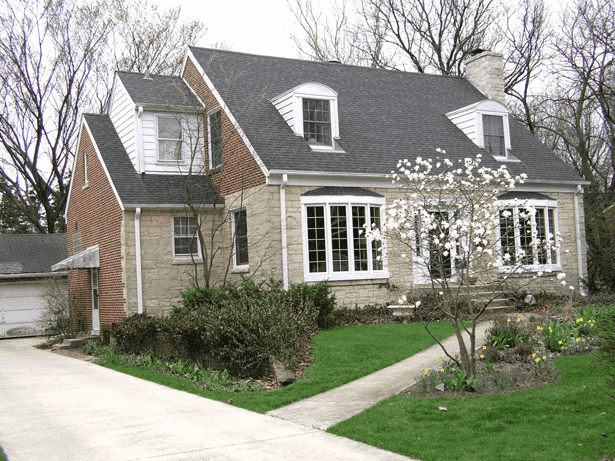}
    \end{subfigure}%
    \begin{subfigure}[htbp]{0.12\textwidth}
        \includegraphics[width=0.97\textwidth, keepaspectratio]{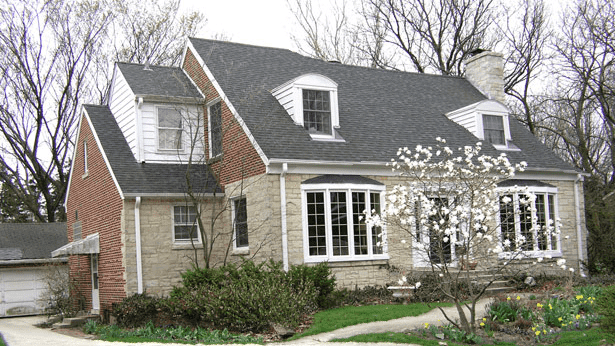}
    \end{subfigure}%
    \begin{subfigure}[htbp]{0.12\textwidth}
        \includegraphics[width=0.97\textwidth, keepaspectratio]{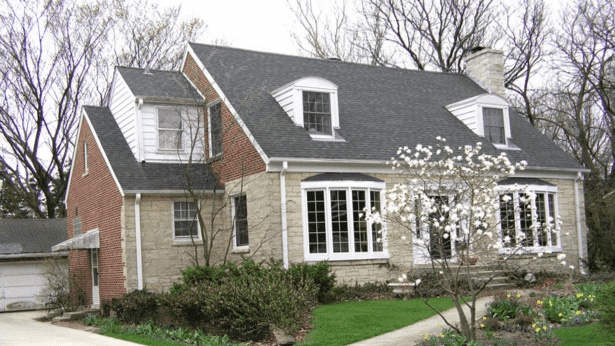}
    \end{subfigure}%
    \begin{subfigure}[htbp]{0.12\textwidth}
        \includegraphics[width=0.97\textwidth, keepaspectratio]{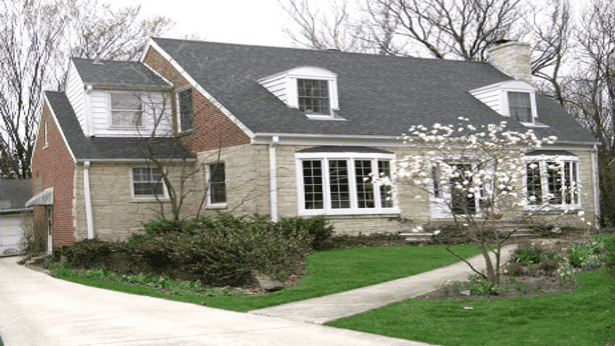}
    \end{subfigure}%
    \begin{subfigure}[htbp]{0.12\textwidth}
        \includegraphics[width=0.97\textwidth, keepaspectratio]{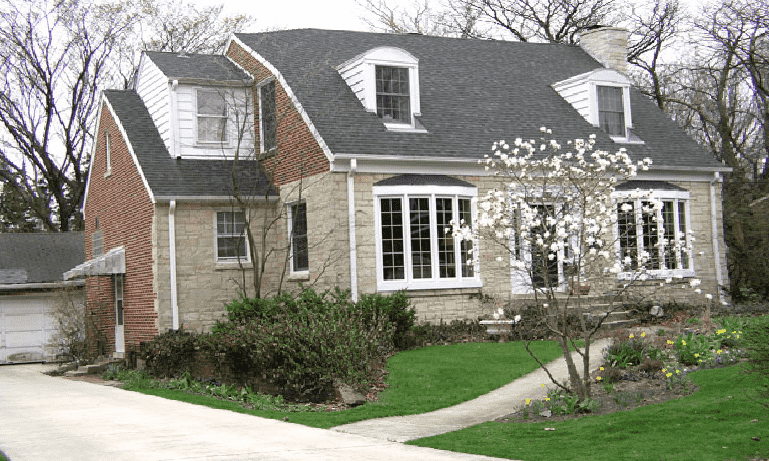}
    \end{subfigure}%
    \begin{subfigure}[htbp]{0.12\textwidth}
        \includegraphics[width=0.97\textwidth, keepaspectratio]{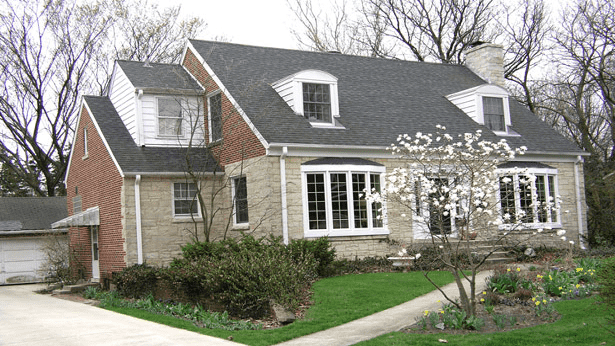}
    \end{subfigure}%
    \begin{subfigure}[htbp]{0.12\textwidth}
        \includegraphics[width=0.97\textwidth, keepaspectratio]{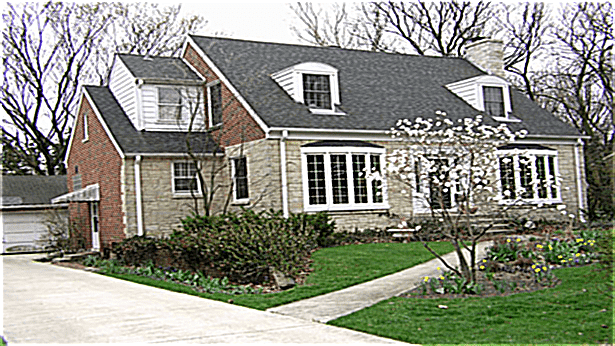}
    \end{subfigure}%
    \begin{subfigure}[htbp]{0.12\textwidth}
        \includegraphics[width=0.97\textwidth, keepaspectratio]{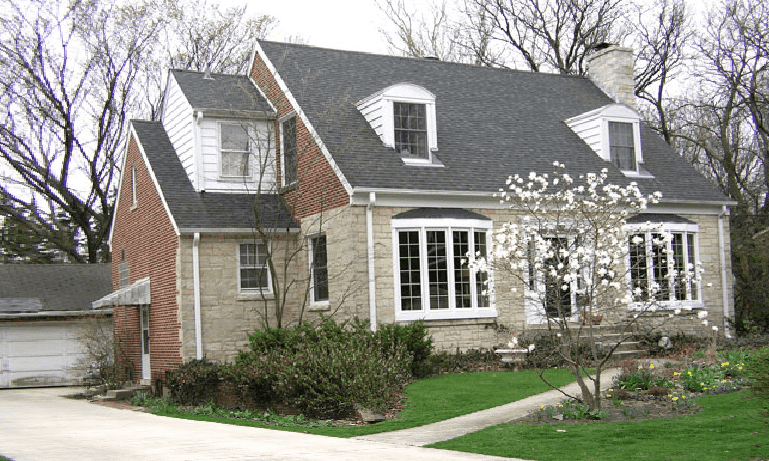}
    \end{subfigure}%
    \vspace{\smallskipamount} 

    \makebox[\textwidth][c]{(b) Family}\\
    \begin{subfigure}[htbp]{0.12\textwidth}
        \includegraphics[width=0.97\textwidth, keepaspectratio]{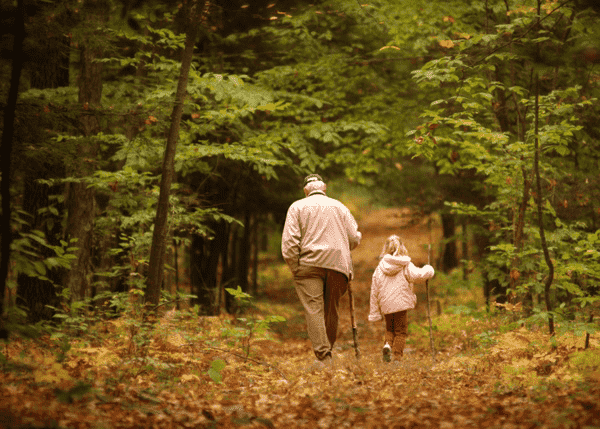}
    \end{subfigure}%
    \begin{subfigure}[htbp]{0.12\textwidth}
        \includegraphics[width=0.97\textwidth, keepaspectratio]{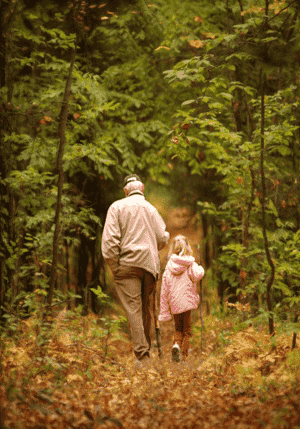}
    \end{subfigure}%
    \begin{subfigure}[htbp]{0.12\textwidth}
        \includegraphics[width=0.97\textwidth, keepaspectratio]{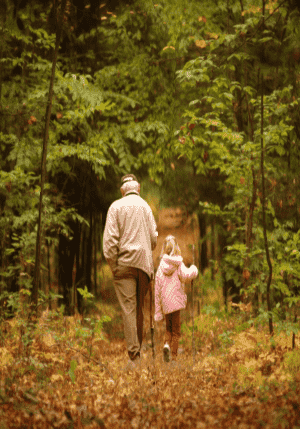}
    \end{subfigure}%
    \begin{subfigure}[htbp]{0.12\textwidth}
        \includegraphics[width=0.97\textwidth, keepaspectratio]{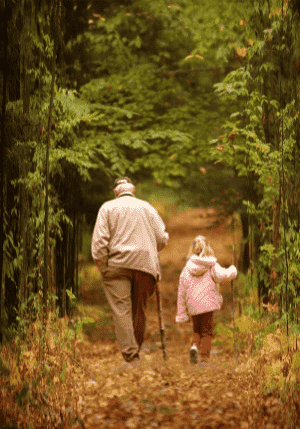}
    \end{subfigure}%
    \begin{subfigure}[htbp]{0.12\textwidth}
        \includegraphics[width=0.97\textwidth, keepaspectratio]{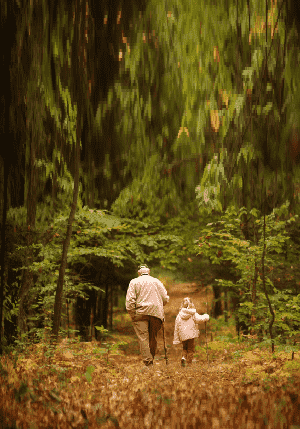}
    \end{subfigure}%
    \begin{subfigure}[htbp]{0.12\textwidth}
        \includegraphics[width=0.97\textwidth, keepaspectratio]{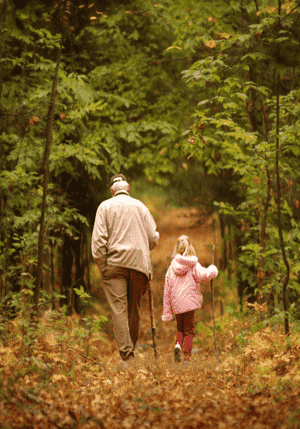}
    \end{subfigure}%
    \begin{subfigure}[htbp]{0.12\textwidth}
        \includegraphics[width=0.97\textwidth, keepaspectratio]{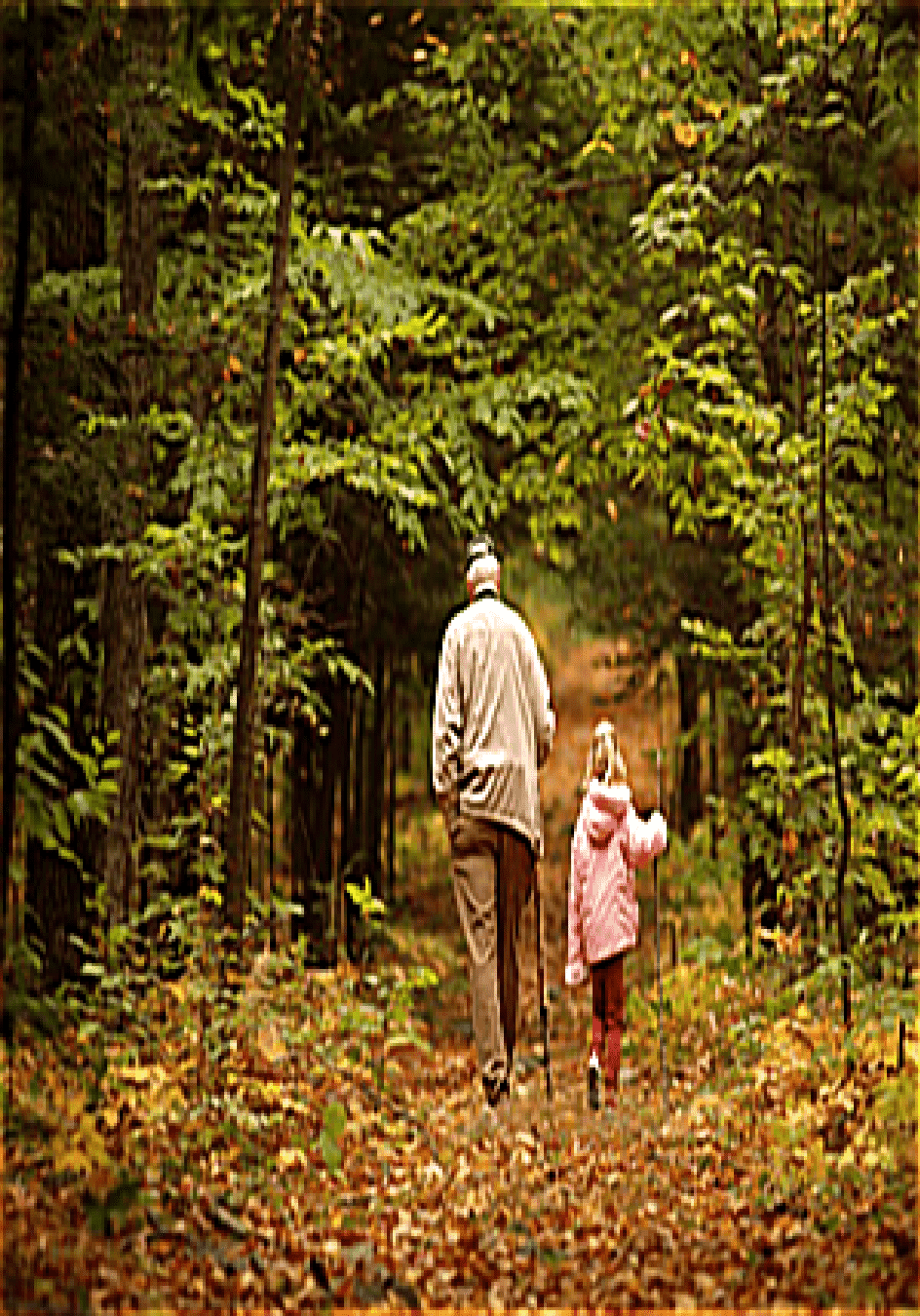}
    \end{subfigure}%
    \begin{subfigure}[htbp]{0.12\textwidth}
        \includegraphics[width=0.97\textwidth, keepaspectratio]{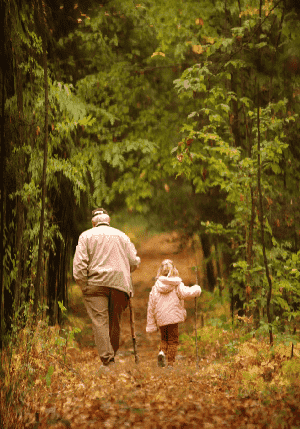}
    \end{subfigure}%
    \vspace{\smallskipamount} 

    \makebox[\textwidth][c]{(c) Penguins}\\
    \begin{subfigure}[htbp]{0.12\textwidth}
        \includegraphics[width=0.97\textwidth, keepaspectratio]{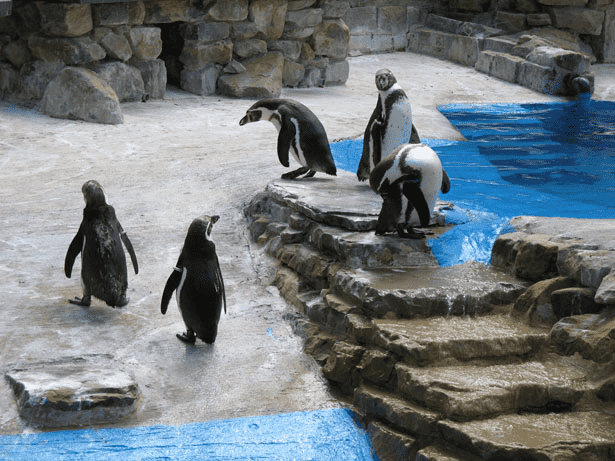}
    \end{subfigure}%
    \begin{subfigure}[htbp]{0.12\textwidth}
        \includegraphics[width=0.97\textwidth, keepaspectratio]{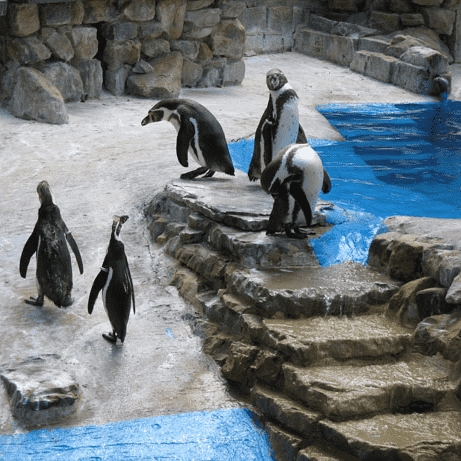}
    \end{subfigure}%
    \begin{subfigure}[htbp]{0.12\textwidth}
        \includegraphics[width=0.97\textwidth, keepaspectratio]{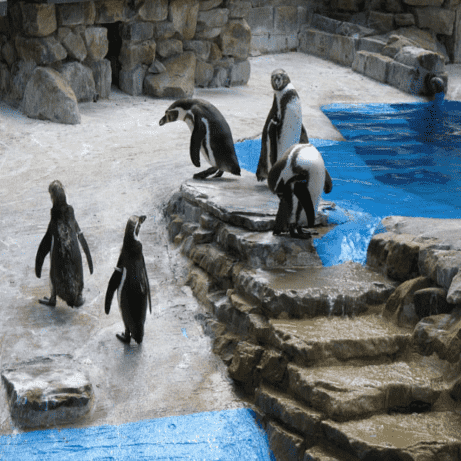}
    \end{subfigure}%
    \begin{subfigure}[htbp]{0.12\textwidth}
        \includegraphics[width=0.97\textwidth, keepaspectratio]{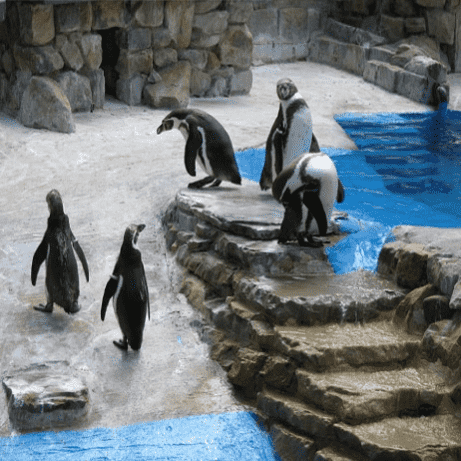}
    \end{subfigure}%
    \begin{subfigure}[htbp]{0.12\textwidth}
        \includegraphics[width=0.97\textwidth, keepaspectratio]{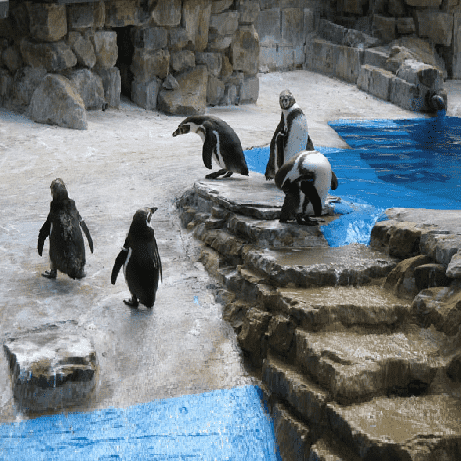}
    \end{subfigure}%
    \begin{subfigure}[htbp]{0.12\textwidth}
        \includegraphics[width=0.97\textwidth, keepaspectratio]{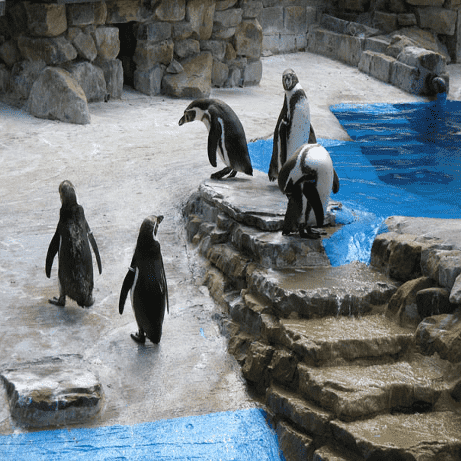}
    \end{subfigure}%
    \begin{subfigure}[htbp]{0.12\textwidth}
        \includegraphics[width=0.97\textwidth, keepaspectratio]{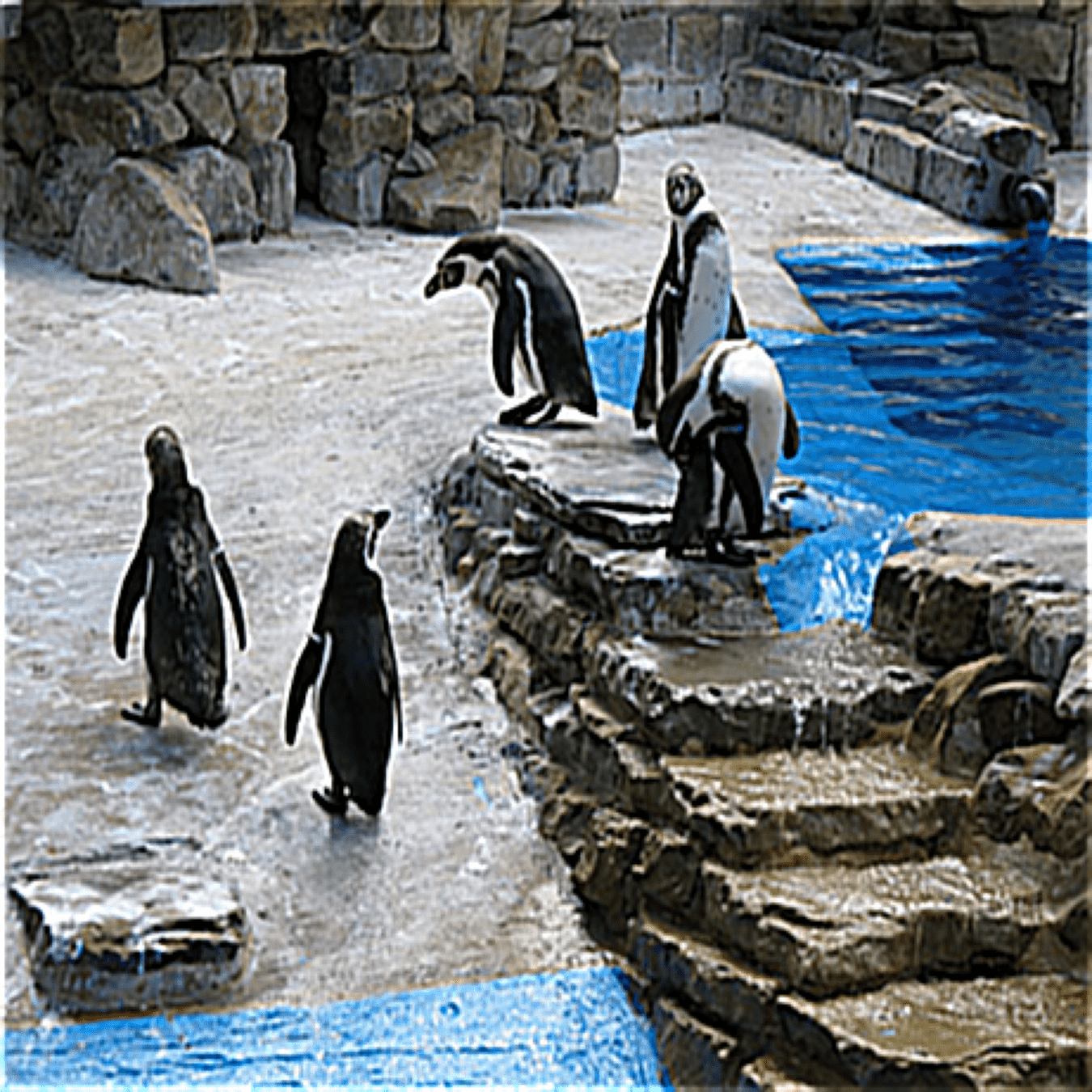}
    \end{subfigure}%
    \begin{subfigure}[htbp]{0.12\textwidth}
        \includegraphics[width=0.97\textwidth, keepaspectratio]{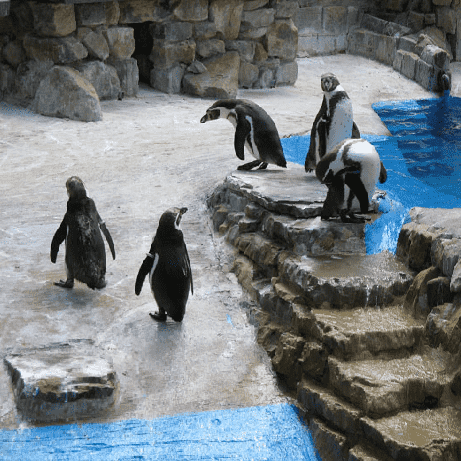}
    \end{subfigure}%
    \vspace{\smallskipamount} 

    \makebox[\textwidth][c]{(d) San Francisco}\\
    \begin{subfigure}[htbp]{0.12\textwidth}
        \includegraphics[width=0.97\textwidth, keepaspectratio]{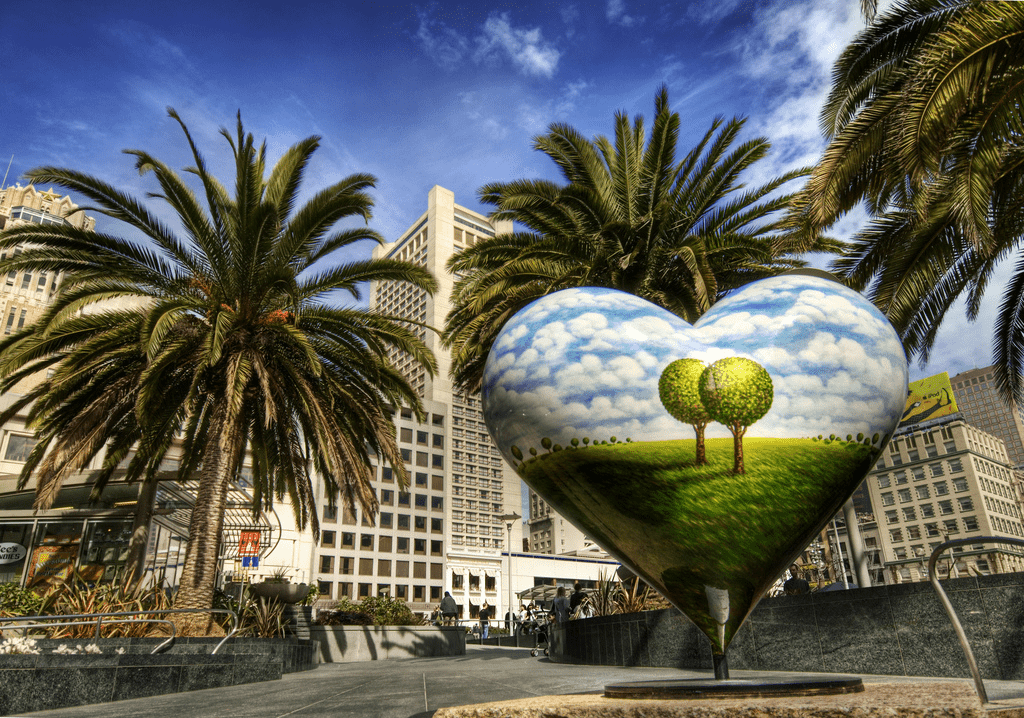}
    \end{subfigure}%
    \begin{subfigure}[htbp]{0.12\textwidth}
        \includegraphics[width=0.97\textwidth, keepaspectratio]{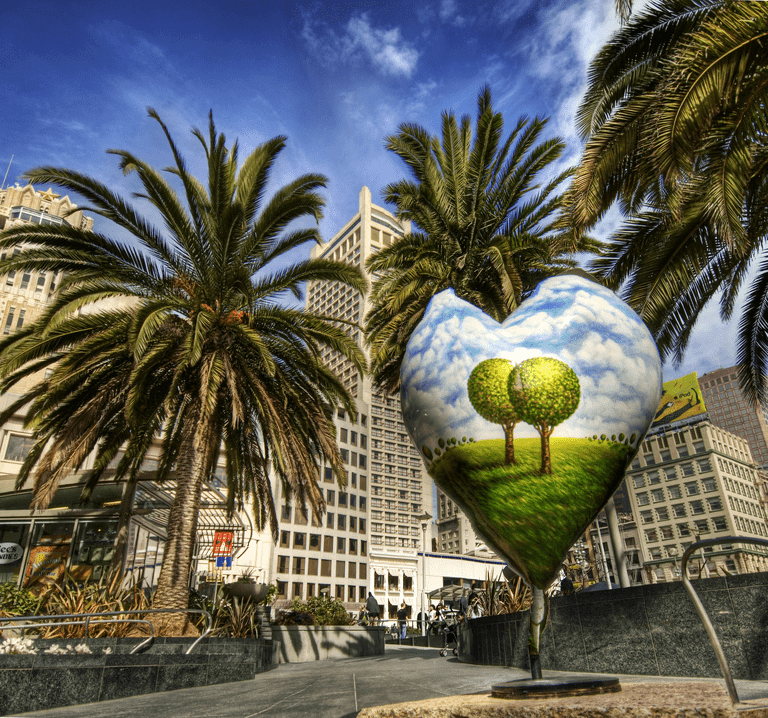}
    \end{subfigure}%
    \begin{subfigure}[htbp]{0.12\textwidth}
        \includegraphics[width=0.97\textwidth, keepaspectratio]{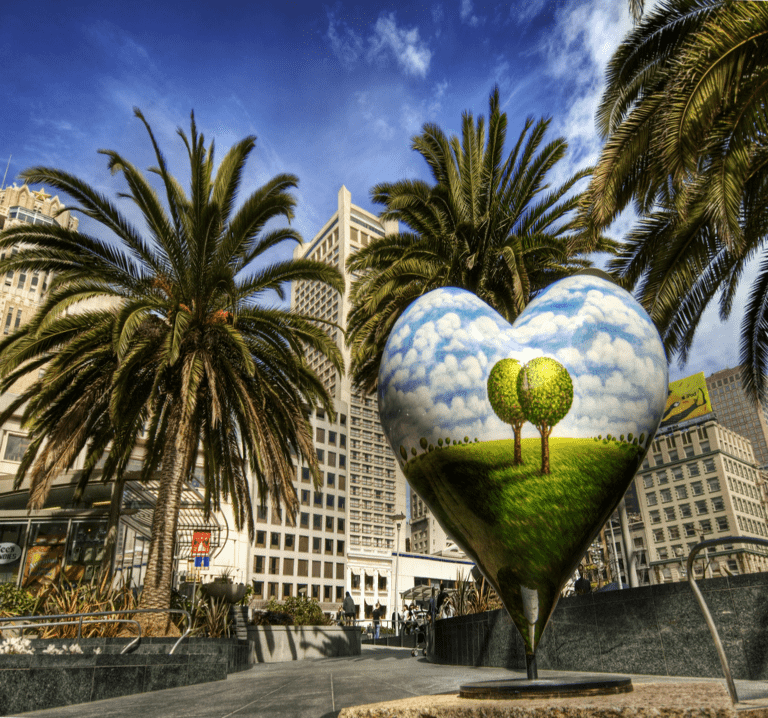}
    \end{subfigure}%
    \begin{subfigure}[htbp]{0.12\textwidth}
        \includegraphics[width=0.97\textwidth, keepaspectratio]{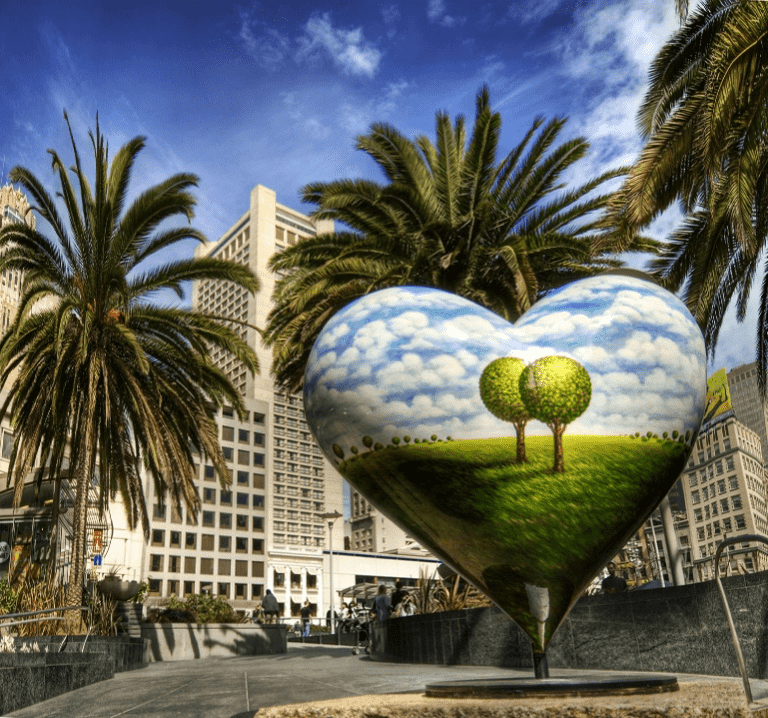}
    \end{subfigure}%
    \begin{subfigure}[htbp]{0.12\textwidth}
        \includegraphics[width=0.97\textwidth, keepaspectratio]{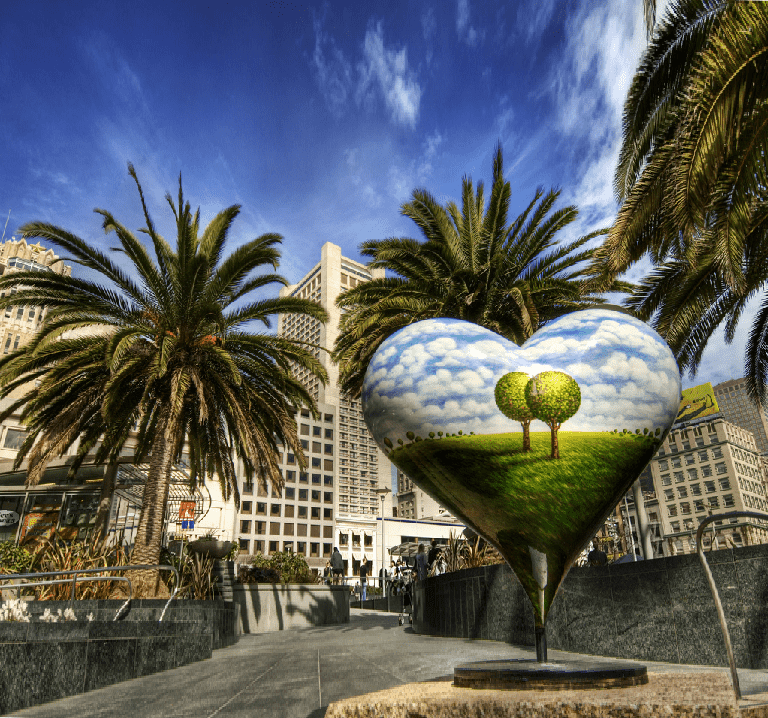}
    \end{subfigure}%
    \begin{subfigure}[htbp]{0.12\textwidth}
        \includegraphics[width=0.97\textwidth, keepaspectratio]{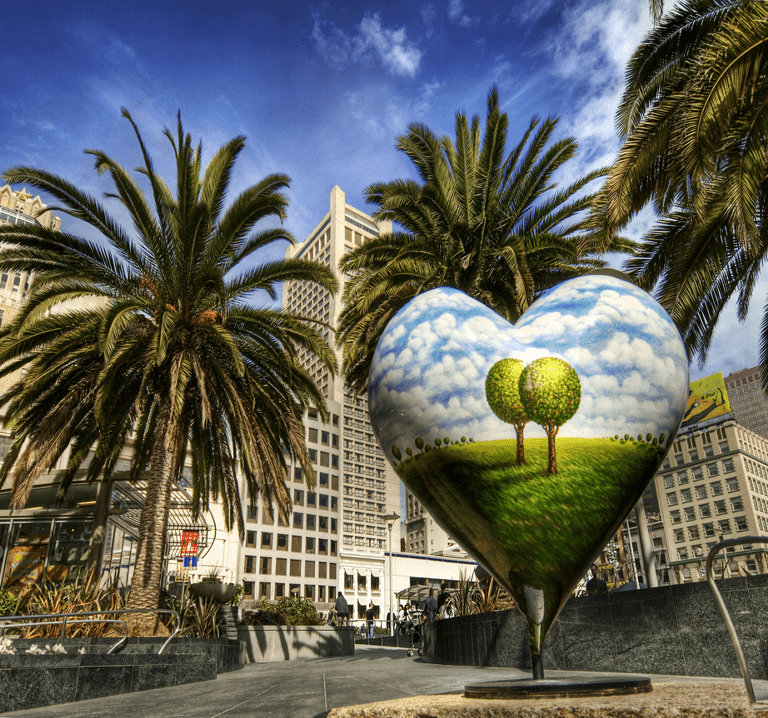}
    \end{subfigure}%
    \begin{subfigure}[htbp]{0.12\textwidth}
        \includegraphics[width=0.97\textwidth, keepaspectratio]{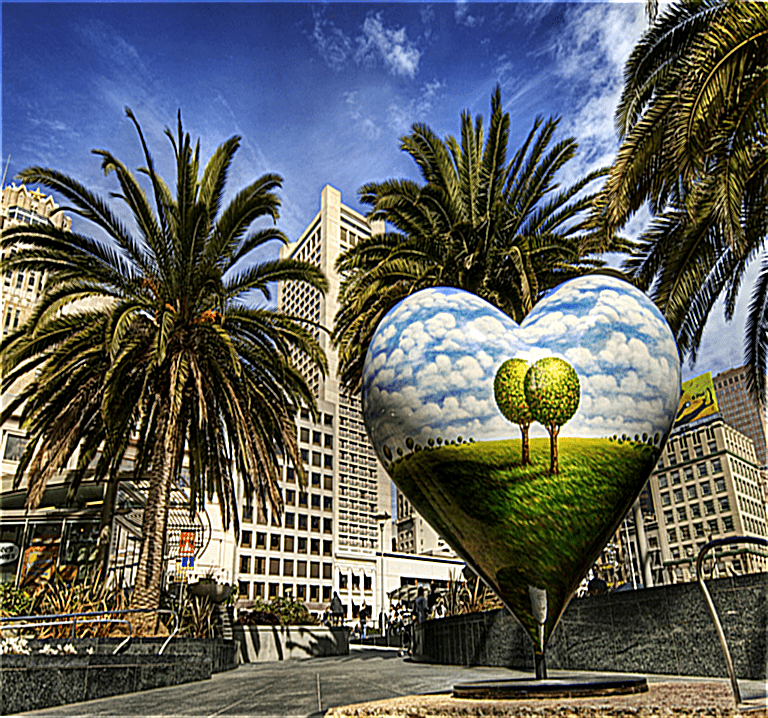}
    \end{subfigure}%
    \begin{subfigure}[htbp]{0.12\textwidth}
        \includegraphics[width=0.97\textwidth, keepaspectratio]{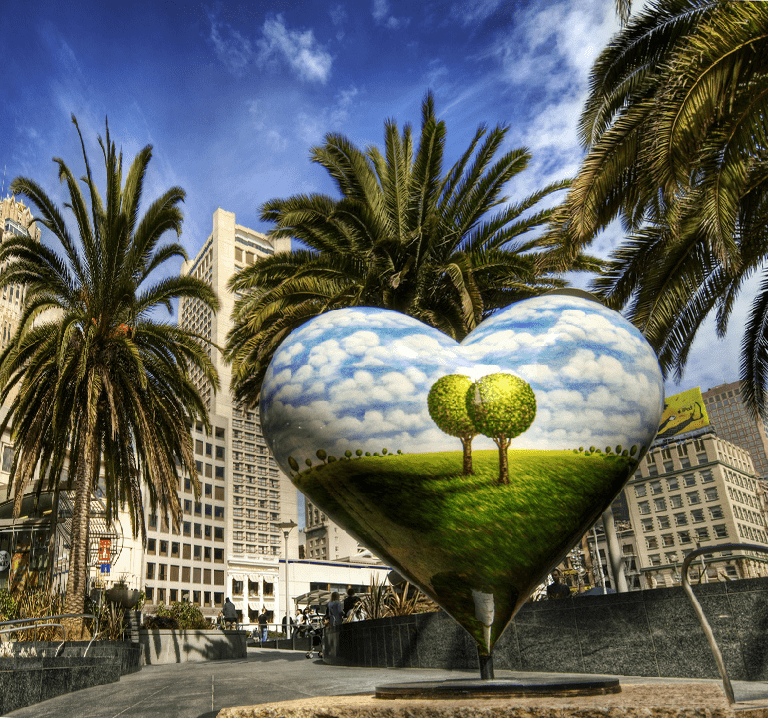}
    \end{subfigure}%
    \vspace{\smallskipamount} 

    \makebox[\textwidth][c]{(d) Sailing}\\
    \begin{subfigure}[htbp]{0.12\textwidth}
        \includegraphics[width=0.97\textwidth, keepaspectratio]{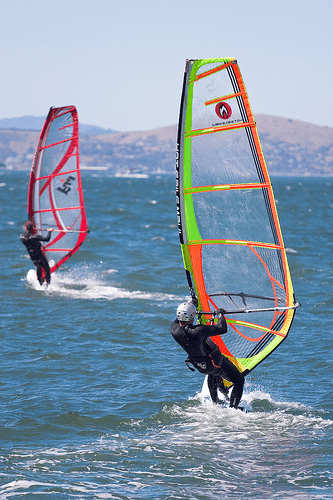}
    \end{subfigure}%
    \begin{subfigure}[htbp]{0.12\textwidth}
        \includegraphics[width=0.97\textwidth, keepaspectratio]{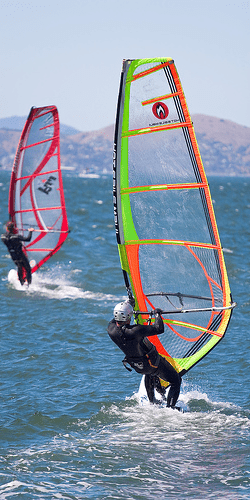}
    \end{subfigure}%
    \begin{subfigure}[htbp]{0.12\textwidth}
        \includegraphics[width=0.97\textwidth, keepaspectratio]{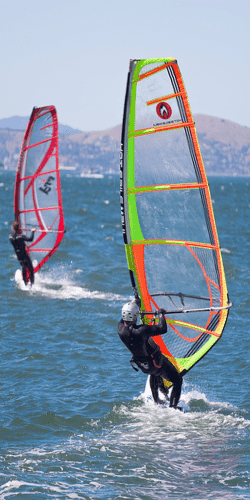}
    \end{subfigure}%
    \begin{subfigure}[htbp]{0.12\textwidth}
        \includegraphics[width=0.97\textwidth, keepaspectratio]{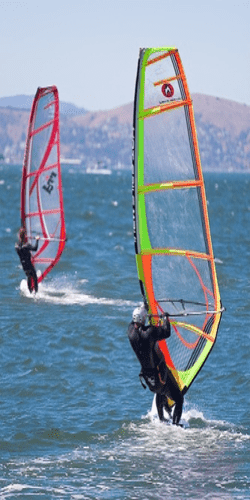}
    \end{subfigure}%
    \begin{subfigure}[htbp]{0.12\textwidth}
        \includegraphics[width=0.97\textwidth, keepaspectratio]{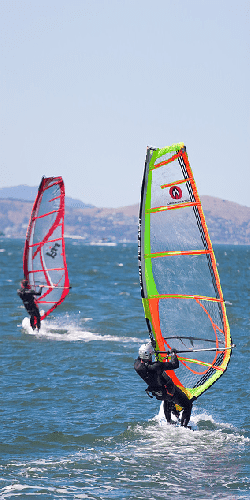}
    \end{subfigure}%
    \begin{subfigure}[htbp]{0.12\textwidth}
        \includegraphics[width=0.97\textwidth, keepaspectratio]{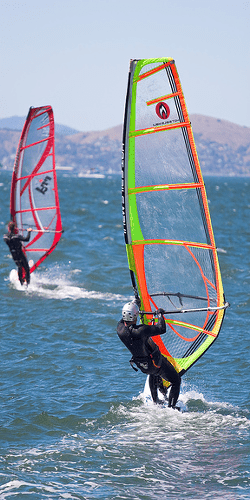}
    \end{subfigure}%
    \begin{subfigure}[htbp]{0.12\textwidth}
        \includegraphics[width=0.97\textwidth, keepaspectratio]{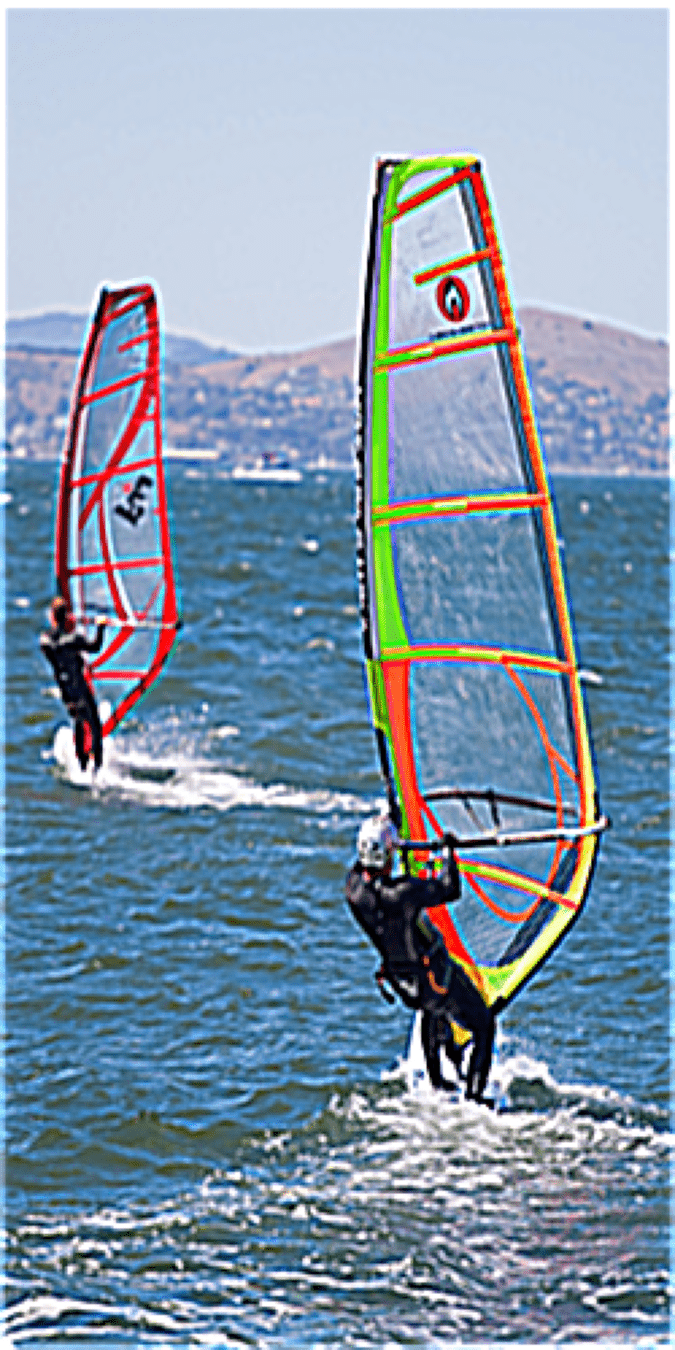}
    \end{subfigure}%
    \begin{subfigure}[htbp]{0.12\textwidth}
        \includegraphics[width=0.97\textwidth, keepaspectratio]{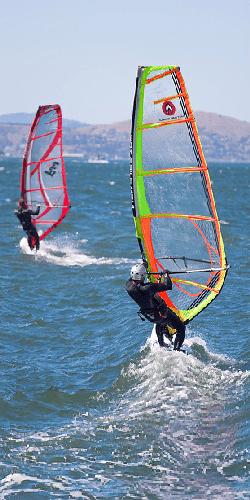}
    \end{subfigure}%
    \vspace{\smallskipamount} 

    \makebox[\textwidth][c]{(e) Bridge}\\
    \begin{subfigure}[htbp]{0.12\textwidth}
        \includegraphics[width=0.97\textwidth, keepaspectratio]{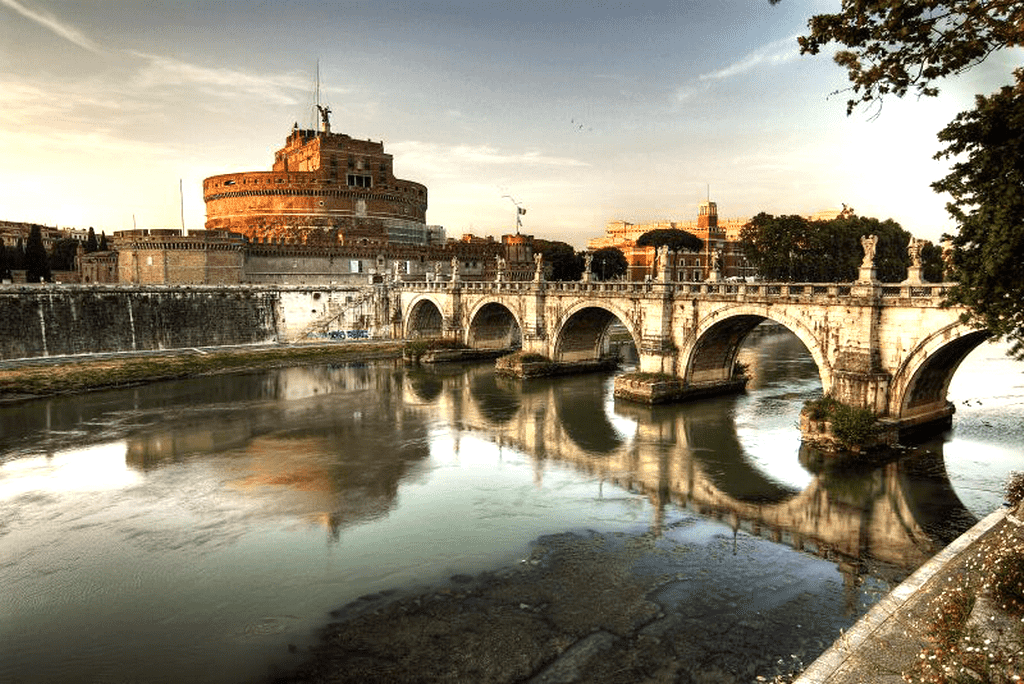}
        \caption*{OI}
    \end{subfigure}%
    \begin{subfigure}[htbp]{0.12\textwidth}
        \includegraphics[width=0.97\textwidth, keepaspectratio]{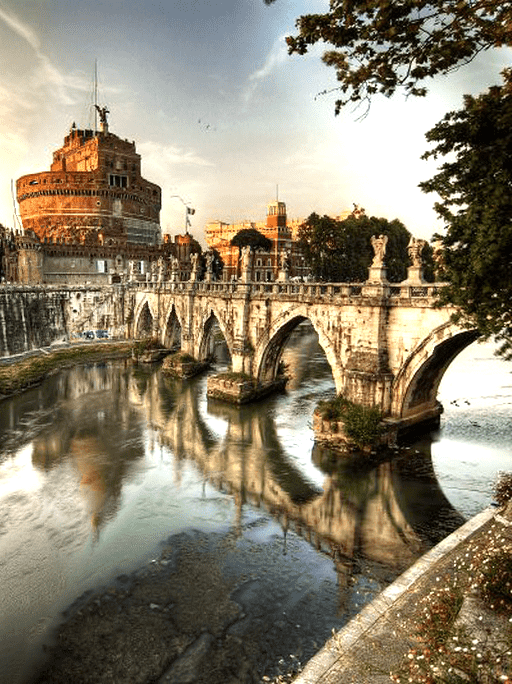}
        \caption*{SC\cite{rubinstein2008improved}}
    \end{subfigure}%
    \begin{subfigure}[htbp]{0.12\textwidth}
        \includegraphics[width=0.97\textwidth, keepaspectratio]{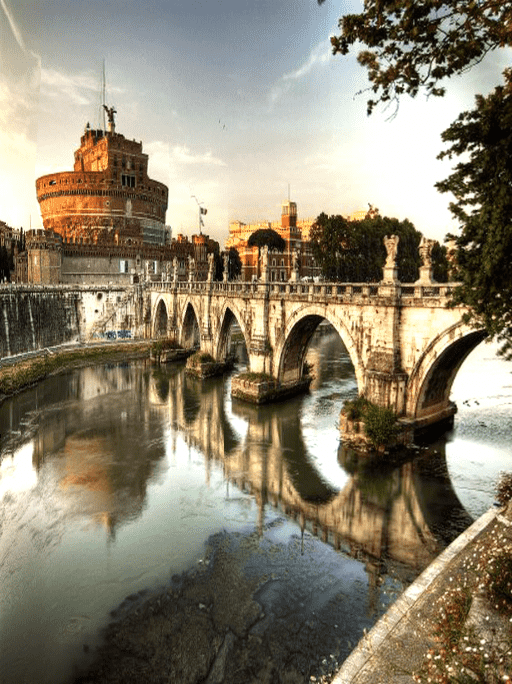}
        \caption*{Warp\cite{wolf2007nonhomogeneous}}
    \end{subfigure}%
    \begin{subfigure}[htbp]{0.12\textwidth}
        \includegraphics[width=0.97\textwidth, keepaspectratio]{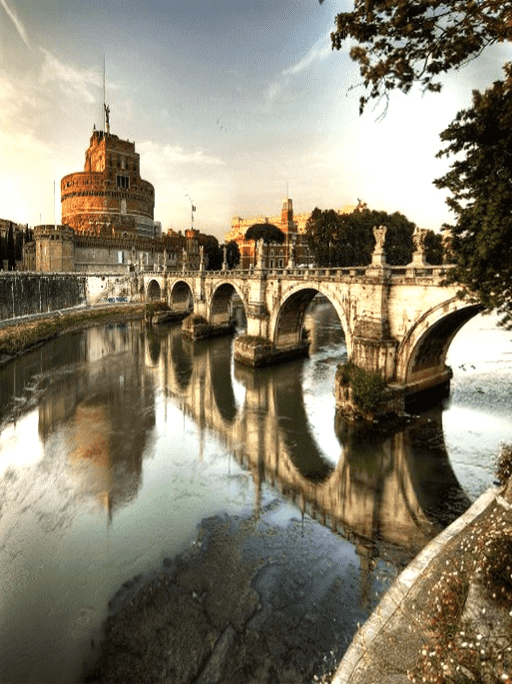}
        \caption*{QP\cite{chen2010content}}
    \end{subfigure}%
    \begin{subfigure}[htbp]{0.12\textwidth}
        \includegraphics[width=0.97\textwidth, keepaspectratio]{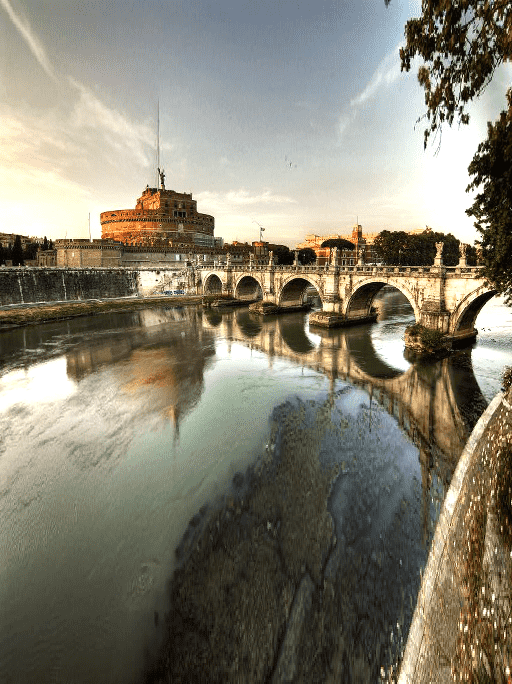}
        \caption*{BR\cite{lau2018image}}
    \end{subfigure}%
    \begin{subfigure}[htbp]{0.12\textwidth}
        \includegraphics[width=0.97\textwidth, keepaspectratio]{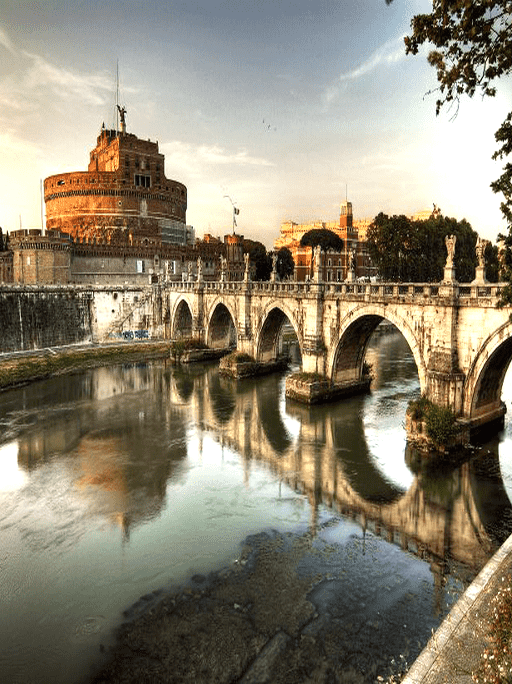}
        \caption*{MO\cite{rubinstein2009multioperator}}
    \end{subfigure}%
    \begin{subfigure}[htbp]{0.12\textwidth}
        \includegraphics[width=0.97\textwidth, keepaspectratio]{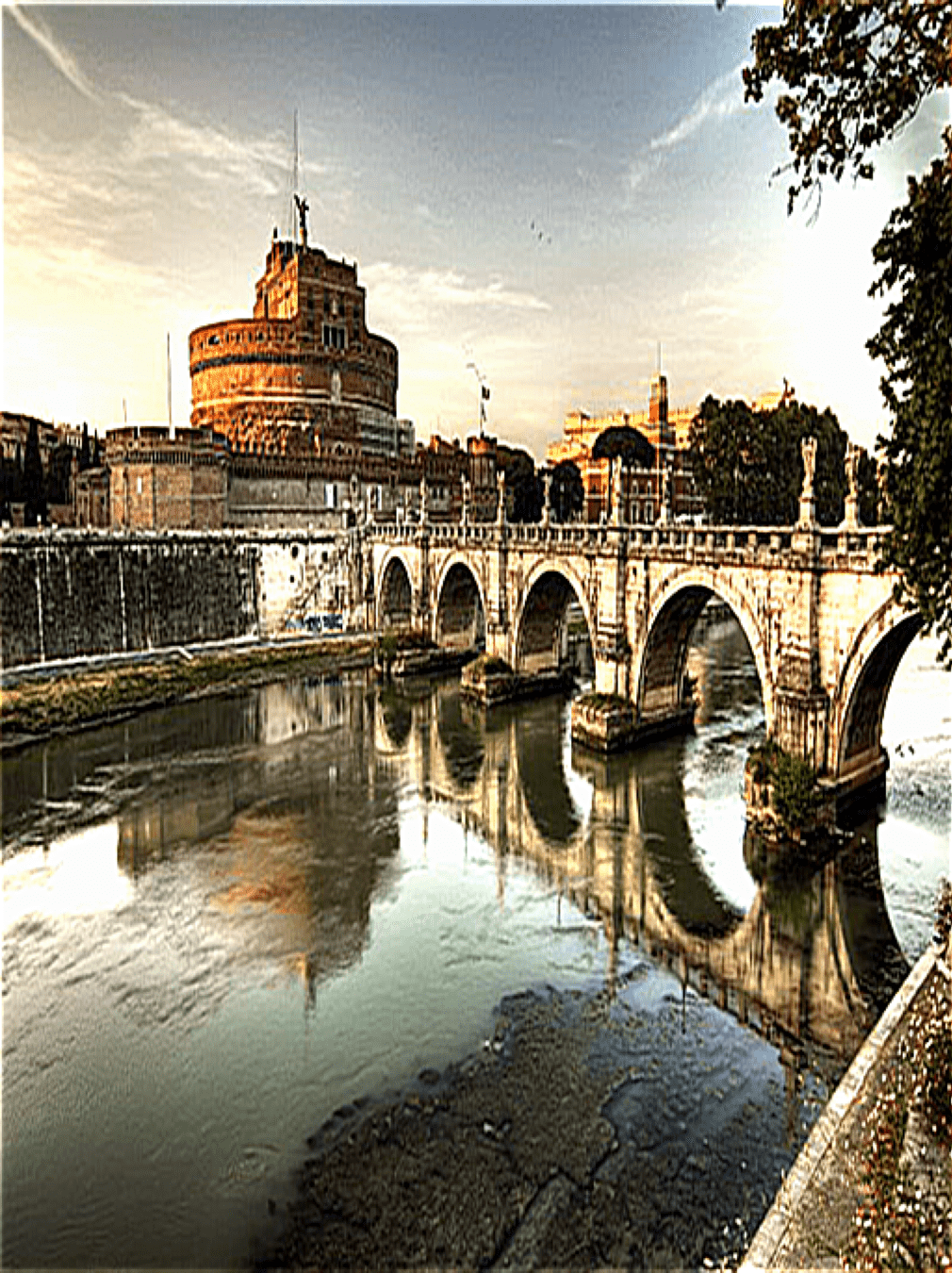}
        \caption*{ML\cite{tu2023muller}}
    \end{subfigure}%
    \begin{subfigure}[htbp]{0.12\textwidth}
        \includegraphics[width=0.97\textwidth, keepaspectratio]{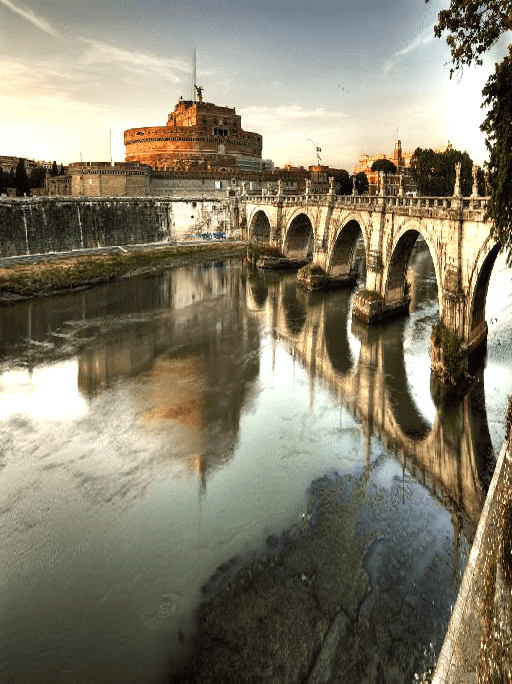}
        \caption*{Our model}
    \end{subfigure}%
    \caption{Comparison of well-known retargeting methods on $RetargetMe$.}
    \label{fig:dataset}
\end{figure*}
  We compare our automatic method with several retargeting algorithms, including seam-carving method (SC) \cite{rubinstein2008improved}, Nonhomogeneous warping (Warp) \cite{wolf2007nonhomogeneous}, image retargeting algorithm via Beltrami representation (BR) \cite{lau2018image}, and multilayer Laplacian resizer (ML) \cite{tu2023muller}. 
  OI represents the original image. 
The $r_{\mathcal{O}}, \mathbf{t}_{\mathcal{O}_i}, R_{\mathcal{L}_j}, \mathbf{t}_{\mathcal{L}_j}, g$ of our retargeting model in these experiments are all computed by (\ref{equ:roi}). 
In Figure \ref{fig:example1}, the SC model shows a noticeable distortion of line structures on the car (circled in red). 
For the Warp model, the white line on the ground is turned into curves (circled in red). 
The car in the results of the ML method looks longer than the original one.
While the car and line structures are well-preserved in the result of our model.

More experimental results are given in Figure \ref{fig:dataset}, in which QP represents Quadratic programming algorithm \cite{chen2010content}, and MO represents Multi-operator media retargeting algorithm \cite{rubinstein2009multioperator}. These results show that, compared with other algorithms, the shape of ROIs and line structures in the results of our model are preserved very well, e.g., the boy with snowman and the Bridge in Figure \ref{fig:dataset}, and the geometric distortion of the background part is also minimized. Since the retargeting mapping of our model is theoretically ensured bijective, all the details in the original images can also be found in the retargeted ones and the image overlapping issues are perfectly avoided.

The convergence of energy minimums during the subdivision process is also checked by experiments (Figure \ref{fig:sub1}). The vertices of the original mesh in Figure \ref{fig:sub1} are a $5\times 4$ lattice. The subdivision method is the corner-chopped subdivision (each simplex is divided into four congruent triangles). We can see that the conformal energy of the warping map is monotonically decreasing with respect to $i$. The energy of both experiments tends to converge. 
\begin{figure*}[htbp]
  \centering
  \caption{Resizing a clownfish image to 50\% of the original width.}
  \begin{subfigure}[htbp]{0.35\textwidth}
    \includegraphics[width=0.97\textwidth, keepaspectratio]{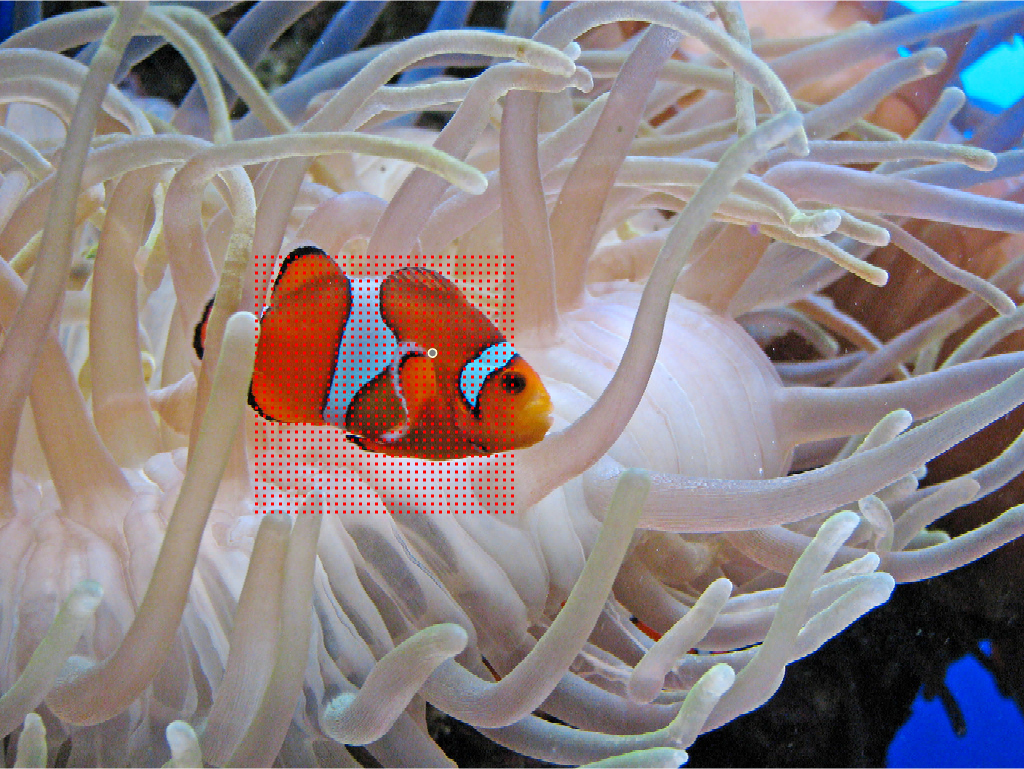}
    \caption{Original image (ROI in red)}
  \end{subfigure}%
  \begin{subfigure}[htbp]{0.65\textwidth}
    \includegraphics[width=0.97\textwidth, keepaspectratio]{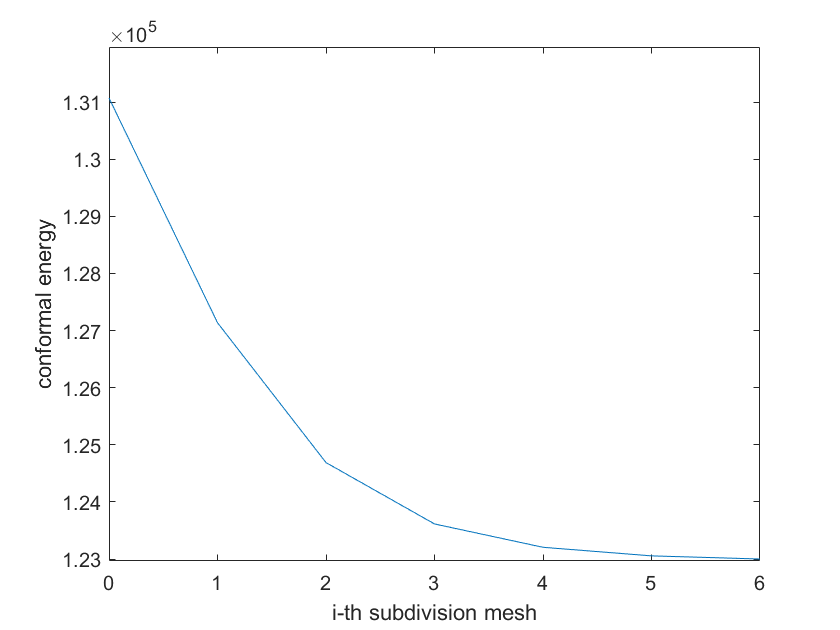}
    \caption{Conformal energy of the warping map on $i^{\text{th}}$ mesh}
  \end{subfigure}%
  \\
  \begin{subfigure}[htbp]{0.5\textwidth}
    \includegraphics[width=0.97\textwidth, keepaspectratio]{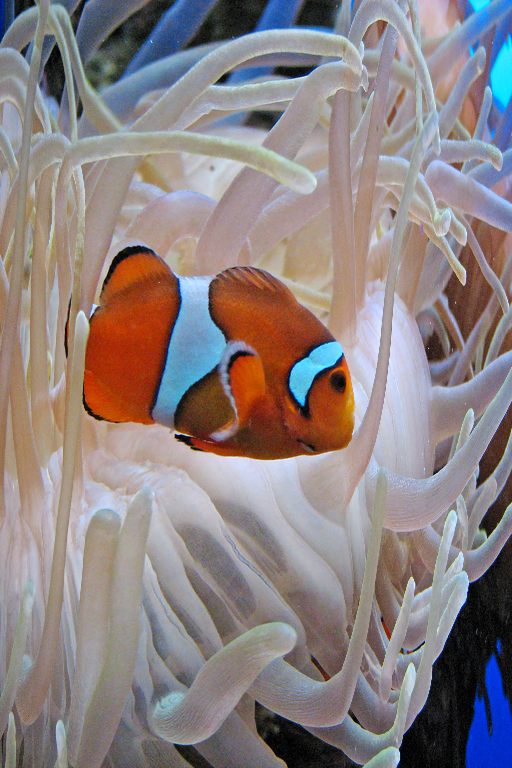}
    \caption{The result on the $0^{\text{th}}$ mesh}
    \end{subfigure}%
  \begin{subfigure}[htbp]{0.5\textwidth}
    \includegraphics[width=0.97\textwidth, keepaspectratio]{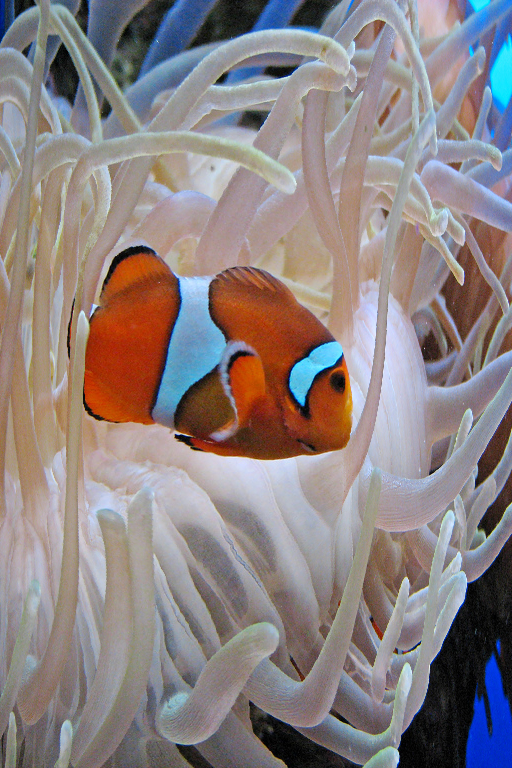}
    \caption{The result on the $6^{\text{th}}$ mesh}
  \end{subfigure}%

\label{fig:sub1}
\end{figure*}

\section{Conclusions}
In this paper, we propose a bijective image retargeting model by minimizing the conformal energy of the warping map, 
which can preserve ROIs and line structures in images. 
We also prove that our model is well-posed and the discrete energy minimizer in our model is a promising approximation of the optimal map if the mesh is dense enough.

Our proofs show the incredible capabilities of simplicial-mesh-based methods in image processing. Under some conditions, any mapping in the $H^1$ space can be well approximated by a simplicial mapping. This proposition will lay a solid theoretical foundation for various simplicial-mesh-based image processing methods.

In this paper, we have discussed the issues related to conformal energy. However, there is a gap in the simplicial mapping approximation of the quasi-conformal energy minimizer because the corresponding matrix of $L^a$ in this paper is not an M-matrix on some meshes.  
Therefore, it is meaningful to research what meshes can be applied to quasi-conformal energy-minimizing problems. Video retargeting is also popular these days. We may extend this algorithm to this problem in the future.

\bibliographystyle{siam}
\bibliography{references}

\end{document}